\documentclass[11pt,a4paper]{article} % replace ECP by EJP if needed.
\usepackage[utf8]{inputenc}

\usepackage{enumerate}  % uncomment to use this package
\usepackage[english]{babel}
\usepackage{amsmath}
\usepackage{amsthm}
\usepackage{amsfonts}
\usepackage{amssymb,bm}
\usepackage{mathrsfs}
\usepackage{stmaryrd}

\usepackage{graphicx}

\usepackage{comment}

\usepackage[expansion]{microtype}
\usepackage{stmaryrd}

\usepackage[pdftex,colorlinks=true]{hyperref}

\hypersetup{
    pdfauthor   = {Enriquez Faraud Ménard Noiry},%
    pdftitle    = {DFS supercritical Configuration Model}}

\usepackage[font={small,sf}, labelfont={sf,bf}, margin=1cm]{caption}

\newtheorem{lemma}{Lemma}
\newtheorem{proposition}{Proposition}
\newtheorem{definition}{Definition}
\newtheorem{theorem}{Theorem}
\newtheorem{corollary}{Corollary}
\newtheorem{remark}{Remark}
\newtheorem*{acks}{Acknowledgements}

\newcommand\E{ \mathbb{E} }

\newcommand\PP{ \mathbb{P} }

\title{Depth First Exploration of a Configuration Model}
\author{Nathana\"el Enriquez, Gabriel Faraud, Laurent Ménard and Nathan Noiry}
\date{}

\newcommand\geom{\mathfrak{e}}

\begin{document}
\maketitle

\begin{abstract}
We introduce an algorithm that constructs a random graph with prescribed degree sequence together with a depth first exploration of it. In the so-called supercritical regime where the graph contains a giant component, we prove that the renormalized contour process of the Depth First Search Tree has a deterministic limiting profile that we identify. The proof goes through a detailed analysis of the evolution of the empirical degree distribution of unexplored vertices. This evolution is driven by an infinite system of differential equations which has a unique and explicit solution. As a byproduct, we deduce the existence of a macroscopic simple path and get a lower bound on its length.
\end{abstract}

%%%%%%%%%%%%%%%%%%%%%%%%%%%%%%%%%%%%%%%%%%%%%%%%%%%%%%%%%%%%%%%%%%%
%%                                                               %%
%% No need for \maketitle.                                       %%
%%                                                               %%
%%%%%%%%%%%%%%%%%%%%%%%%%%%%%%%%%%%%%%%%%%%%%%%%%%%%%%%%%%%%%%%%%%%

%%%%%%%%%%%%%%%%%%%%%%%%%%%%%%%%%%%%%%%%%%%%%%%%%%%%%%%%%%%%%%%%%%%
%%                                                               %%
%% Please replace what follows by the body of your article       %%
%% (up to the bibliography):                                     %%
%%                                                               %%
%%%%%%%%%%%%%%%%%%%%%%%%%%%%%%%%%%%%%%%%%%%%%%%%%%%%%%%%%%%%%%%%%%%

\section{Introduction}

Historically, the configuration model was introduced by Bender and Canfield \cite{BenderCanfield78}, Bollob\'as \cite{Bollobas80} and Wormald \cite{wormald1980some} as a random multigraph with $N$ vertices and prescribed degree sequence $d_1,\ldots,d_N$. It turns out that this model shares a lot of features with the Erd\H{o}s-Rényi random graph. In particular it exhibits a phase transition for the existence of a unique macroscopic connected component. This phase transition, as well as the size of this so-called giant component, was studied in detail in \cite{MolloyReed95,MolloyReed98,JansonLuczak2009}. The proof of these results relies on the analysis of a construction algorithm which takes as input a collection of $N$ vertices having respectively $d_1, \ldots, d_N$ half-edges coming out of them, and returns as output a random multigraph with degree sequence $d_1, \ldots, d_N$, by connecting step by step the half-edges. The way \cite{MolloyReed95,MolloyReed98,JansonLuczak2009,MR3343756} connect these half-edges is as follows: at a given step in this algorithm, a uniform half-edge of the growing cluster is connected to a uniform not yet connected half-edge.

In this paper, we introduce a construction algorithm which, in addition to constructing the configuration model, provides an exploration of it. This exploration corresponds to the Depth First Search algorithm which is roughly a nearest neighbor walk on the vertices that greedily tries to go as deep as possible in the graph. The output of the Depth First Search Algorithm is a spanning rooted plane tree for each connected component of the graph, whose height provides a lower bound on the length of the largest simple path in the corresponding component.

A similar exploration (namely, a breadth-first exploration) has been successfuly used by Aldous \cite{Aldous1997} for the Erd\H{o}s-Rényi model in the critical window where the connected componenents are of polynomial size. The structure of the graph in this window was further studied in \cite{AddarioBroutinGold2012}. For the configuration model, a similar critical window was also identified and studied. See \cite{HM,Riordan2012,MR3343756,DvdHLS}.

\bigskip
 
The purpose of this article is to study this algorithm on a supercritical configuration model and in particular the limiting shape of the contour process of the tree associated to the Depth First exploration of the giant component. Unlike in the previous construction of \cite{MolloyReed95,MolloyReed98,JansonLuczak2009,MR3343756}, where the authors only studied the evolution of the empirical distribution of the degree of the unexplored vertices, we have to deal with the empirical distribution of the degree of the unexplored vertices in the graph that they induce inside the final graph. The analysis of this evolution is much more delicate and is in fact the heart of our work, this is the content of Proposition \ref{prop:law}.

It turns out that a step by step analysis of the construction does not work. Still, it is possible to track, at some ladder times, the evolution of the degrees of the unexplored vertices in the graph they induce. In this time scale, using a generalization of the celebrated differential equations method of Wormald \cite{Wormald} provided in the appendix (see also \cite{Warnke} for a recent article on this method), we are able to show that the evolution of the empirical degree distribution of the unexplored vertices has a fluid limit which is driven by an infinite system of differential equations. This system as such cannot be handled. We have to introduce a time change which, surprisingly, corresponds to the proportion of explored vertices, in term of the construction algorithm. Another surprise is that the resulting new system of differential equations admits an explicit solution through the generating series they form. In order to apply Wormald's method, we need to establish the uniqueness of this solution. This task, presented in Section \ref{sec: EqDiff}, is also intricate and is based on the knowledge of the explicit solution mentioned above.

Combining Proposition \ref{prop:law} with an analysis of the ladder times, we prove that the renormalized contour process of the spanning tree of the Depth First Search algorithm converges to a deterministic profile for which we give an explicit parametric representation. This is the object of Theorem \ref{th:profile}. A direct consequence is a lower bound on the length of the longest simple path in a supercritical model, see Corollary \ref{coro:length}. To the best of our knowledge, this lower bound seems to be the best available for a generic initial degree distribution. The only other generic bound for configuration models that we could find is due to Frieze and Jackson \cite{FriezeJackson87} in the setting where the graphs have bounded degrees. They establish a lower bound on the longest induced path. However, this bound vanishes as the largest degree tends to infinity.

We do not believe that our bound is sharp. The question of the length of the longest simple path in a configuration model is actually still open in generic cases. To the best of our knowledge, the only solved cases are $d$-regular random graphs that are known to be (almost) Hamiltonian \cite{Bollobas1983}. However, a main advantage of our bound is that it is given by an explicit construction in linear time, which is not the case for the regular graphs setting. For additional details and references on the Erd\H{o}s-Rényi setting, we refer to the survey \cite{Ksurvey}  and to the article \cite{AF}.

Let us mention that the ingredient of ladder times, used in the proof of Theorem \ref{th:profile}, was already present in the context of Erd\H{o}s-R\'enyi graphs in \cite{EnriquezFaraudMenard20}. The novelty and core of the present article is the analysis of the empirical degree distribution of the unexplored vertices at the ladder times, which was straightforward in the case of Erd\H{o}s-R\'enyi graphs as it is in that case, along the construction, a Binomial distribution with decreasing parameter.

In order to illustrate our results, we provide explicit computations together with simulations in the setting where the initial degree distribution follows respectively a Poisson law (recovering results of \cite{EnriquezFaraudMenard20} in the Erd\H{o}s-Rényi setting), a Dirac mass at $d\geq 3$ (corresponding to $d$-regular random graphs) and a Geometric law. We also discuss briefly the heavy tailed case which also falls into the scope of our results.

\section{Definition of the DFS exploration and main results} 

\subsection{The Depth First Search algorithm} \label{sec:defnDFS}

Consider a multigraph $\mathrm{G} = (\mathrm V, \mathrm E)$ whit vertex set $\mathrm V = \{1, \ldots, N \}$. The DFS exploration of $\mathrm G$ is the following algorithm.

For every step $n$ we consider the following objects, defined by induction.
\begin{itemize}
\item $A_n$, the active vertices, is an ordered list of elements of $\mathrm{V}$.
\item $S_n$, the sleeping vertices, is a subset of $\mathrm{V}$. This subset will never contain a vertex of $A_n$.
\item $R_n$, the retired vertices, is another subset of $\mathrm{V}$ composed of all the vertices that are neither in $A_n$ nor $S_n$.
\end{itemize}
At time $n=0$, choose a vertex $v$ uniformly at random. Set:
\[
\begin{cases}
A_0 &= (v) , \\
S_0 &= \mathrm{V}_N \setminus \{v\}, \\
R_0 &= \emptyset.
\end{cases}
\]
Suppose that $A_n$, $S_n$ and $R_n$ are constructed. Three cases are possible.
\begin{enumerate}
\item If $A_n = \emptyset$, the algorithm has just finished exploring a connected component of $\mathrm G$. In that case, we pick a vertex $v_{n+1}$ uniformly at random inside $S_n$ and set:
\[
\begin{cases}
A_{n+1} &= \left(v_{n+1} \right), \\
S_{n+1} &= S_n \setminus \{v_{n+1}\}, \\
R_{n+1} &= R_n.
\end{cases}
\]

\item If $A_n \neq \emptyset$ and if its last element $u$ has no neighbor in $S_n$, the DFS backtracks and we set:
\[
\begin{cases}
A_{n+1} &= A_n - u, \\
S_{n+1} &= S_n,  \\
R_{n+1} &= R_n \cup \{u\}.
\end{cases}
\]

\item If $A_n \neq \emptyset$ and if its last element $u$ has a neighbor in $S_n$, the DFS goes to the smallest neighbor of $u$, say $v$, and we set:
\[
\begin{cases}
A_{n+1} &= A_n + v \\
S_{n+1} &= S_n \setminus \{v\},  \\
R_{n+1} &= R_n.
\end{cases}
\]
\end{enumerate}

This algorithm explores the whole graph and provides a spanning tree of each connected component. In Section \ref{sec:ConstrDFS}, we will provide an algorithm that constructs simultaneously a random graph and a DFS on it.

The algorithm finishes after $2N$ steps. For every $0 \leq n \leq 2N$, we set $X_n := |A_n|$. This walk is called the {\it contour process} associated to the spanning forest of the DFS. In words, it is a $\pm 1$ walk that starts at $X_0 = 0$, stays nonnegative and ends at $X_{2N}=0$, which increases by $1$ each time the DFS moves on (corresponding to point 1. or 2.) and decreases by one each time the DFS backtracks (corresponding to point 3.). Notice that $X_n=0$ when the process starts the exploration of a new connected component. Therefore, each excursion of $(X_n)$ corresponds to a connected component of $\mathrm{G}$. 

\subsection{The Configuration model}

We now turn to the definition of the configuration model.

\begin{definition}
Let $\mathbf{d} = (d_1, \ldots, d_N) \in \mathbb{Z}^N_+$ be such that $d_1 + \cdots + d_N$ is even. We interpret $d_i$ as a number of half-edges attached to vertex i. Then, the configuration model $\mathscr{C}(\mathbf{d})$ associated with the degree sequence $\mathbf{d}$ is the random multigraph with vertex set $\{1, \ldots, N\}$ obtained by a uniform matching of these half-edges. If $d_1 + \cdots + d_N$ is odd, we change $d_N$ into $d_N +1$ and do the same construction.
\end{definition}

We will study sequences of configuration models whose associated sequence of empirical degree distribution converges to a given probability measure.

\begin{definition} \label{def:confasym}
Let $\bm {\pi}$ be a probability distribution on $\mathbb{Z}_+$. For every $N \geq 1$, let $\mathbf{d}^{(N)} = (d^{(N)}_1, \ldots, d^{(N)}_N) \in \mathbb{Z}^N_+$. We say that $( \mathscr{C}(\mathbf{d}^{(N)}))_{N\geq 1}$ has asymptotic degree distribution $\bm \pi$ if
\[  \forall k\geq 0, \quad \frac{1}{N} \sum\limits_{i=1}^N \mathbf{1}_{\{d^{(N)}_i = k\}} \underset{  N \rightarrow + \infty}{ \longrightarrow} \bm \pi (\{k\}). \]
\end{definition} 

As observed in \cite{MolloyReed95}, the configuration model exhibits a phase transition for the existence of a unique macroscopic connected component. In this article, we will restrict our attention to supercritical configuration models, that is where this giant component exists.

\begin{definition}\label{defn:supercritical}
Let $\bm \pi$ be a probability distribution on $\mathbb Z_+$ such that $\sum_{k\geq 0} \bm \pi(\{k\}) k^2  <\infty$ and denote by $f_{\bm \pi}$ its generating function. Let $ \hat{ \bm  \pi}$ be the probability distribution having generating function
\[
\widehat f_{\bm \pi} (s):= f_{ \hat{ \bm  \pi}} (s) = \frac{ f'_{\bm \pi}(s)}{f'_{\bm \pi}(1)}.
\]
We say that $\bm \pi$ is supercritical if $\widehat{f_{\bm \pi}}^\prime (1) > 1$. Notice that, denoting by $D_{\bm \pi}$ a random variable with law $\bm \pi$, it is equivalent to
\[  \frac{\mathbb{E}[ D_{\bm \pi} (D_{\bm \pi} -1) ]}{\mathbb{E}[ D_{\bm \pi} ]} > 1.  \]
In that case we define $\rho_{\bm \pi}$ to be the smallest positive solution of the equation
\[
1 - \rho_{\bm \pi} = \widehat f_{\bm \pi} (1 - \rho_{\bm \pi}).
\]
Finally, we set
\[
\xi_{\bm \pi} := 1 - f_{\bm \pi}(1 - \rho_{\bm \pi}).
\]
\end{definition}
The number $\rho_{\bm \pi}$ is the probability that a Galton-Watson tree with distribution $\widehat{\bm \pi}$ is infinite, whereas the number $\xi_{\bm \pi}$ is the survival probability of a tree where the root has degree distribution $\bm \pi$ and individuals of the next generations have a number of children distributed according to $\widehat{\bm \pi}$. In this article, we study sequence of configuration models $\mathscr{C} ( \mathbf{d}^{(N)} )$ whose asymptotic degree distribution is a supercritical probability measure $\bm \pi$. 

Denoting by $\mathrm{C}_1^{(N)}, \mathrm{C}_2^{(N)}, \ldots$ the sequence of connected components of $\mathcal{C}(\mathbf{d}^{(N)})$ ordered by decreasing number of vertices,  one has
\[  \frac{|\mathrm{C}_1^{(N)}|}{N} \overset{ \mathbb{P}}{\underset{N \rightarrow +\infty}{ \longrightarrow}} \rho_\pi, \]
and the other connected components are microscopic, see for example \cite{MolloyReed95,MolloyReed98,JansonLuczak2009,MR3343756}.

We finally make the two following technical assumptions:
\begin{itemize}
\item The following convergence holds:
\begin{equation} 
\lim\limits_{N \rightarrow + \infty}  \frac{ {d_1^{(N)}}^2 + \cdots + {d_N^{(N)}}^2}{N}  = \sum\limits_{k \geq 0} k^2 {\bm \pi}(\{k\}). \tag{{\bf A1}}
\label{assump: moment}
\end{equation}
\item  There exists $\gamma > 2$ and $C>0$ such that
\begin{equation}
 \max \left\{d_1^{(N)} ,\ldots, d_N^{(N)} \right\} \leq C\, N^{1/\gamma}. \tag{{\bf A2}}
\label{assump: degree}
\end{equation}
\end{itemize}
% \quad \text{and} \quad {\bm \pi}\left(\{ k, k+1, \ldots \}\right) \leq C k^{-\gamma}.
Assumption \eqref{assump: moment} is a classical assumption and is needed to get estimates on the size of the giant component, see \cite{vdH}. Assumption  \eqref{assump: degree} is a (weak) technical assumption needed for our approach and may not be optimal.

\subsection{Main results}

We now state our first result. Define $\alpha\ge 0$ and consider the graph induced by the sleeping vertices after having explored $\lfloor \alpha N \rfloor$ vertices when performing the DFS algorithm on a configuration model. It is clear that this induced graph is also a configuration model. The purpose of the following theorem is to identify its asymptotic degree distribution. It turns out this distribution only depends on $\alpha$ and on the initial degree distribution $\bm \pi$.
\begin{proposition} \label{prop:law}
Let $\bm \pi$ be a supercritical probability measure on $\mathbb{Z}_+$ with generating series $f$ and let $( \mathscr{C}(\mathbf{d}^{(N)}) )_{ N \geq 1}$ be a configuration model with supercritical asymptotic degree distribution $\bm \pi$.  Assume {\bf (A1)} and {\bf (A2)}.

Let $\alpha_c$ be the smallest positive solution of the equation
\[ \frac{f_{\bm \pi}^{''}\left( f_{\bm \pi}^{-1}(1-\alpha)  \right)}{f_{\bm \pi}^{'}(1)} = 1.  \]
For every $\alpha \in [0,\alpha_c]$, let $\bm \pi _{\alpha}$ be the probability distribution on $\mathbb{Z}_+$ with generating series
\[  g(\alpha,s) = \frac{1}{1-\alpha} f_{\bm \pi} \left(  f_{\bm \pi}^{-1}(1-\alpha) - (1-s) \frac{f_{\bm \pi}'(f_{\bm \pi}^{-1}(1-\alpha))}{f_{\bm \pi}'(1)} \right), \]
and write $\tau^{(N)}(\alpha) = \inf \{ k \geq 1: \, |S^{(N)}_k| \leq (1-\alpha)N \}$. Then, conditionally on their degree sequence, the graphs induced by the vertices of $S^{(N)}_{\tau^{(N)}(\alpha)}$ inside $\mathscr{C}(\mathbf{d}^{(N)})$ have the law of configuration models with asymptotic degree distribution $\bm \pi_\alpha$.
\end{proposition}

\begin{remark}
We consider $\alpha$ up to some constant $\alpha_c$, which corresponds to the time when so many vertices have been visited that the remaining graph of sleeping vertices is subcritical. 
\end{remark}

Our main result describes the asymptotic behavior of the contour process $X_n = |A_n|$ of the plane forest constructed by the DFS on a configuration model. 

\begin{theorem} \label{th:profile}
Under the assumptions of Proposition \ref{prop:law}, the following limit holds in probability for the topology of uniform convergence: 
$$ \forall t \in [0,2], \quad \lim_{N\to \infty} \frac{X_{\lceil tN\rceil }}{N}=h(t),$$
where the function $h$ is continuous on $[0, 2]$, null on the interval $[2 \xi_{\bm \pi}, 2]$ and defined below on the interval $[0,2\xi_{\bm \pi}]$.

There exists a unique implicit function $\alpha (\rho)$ defined on $[0,\rho_{\bm \pi}]$ such that $1-\rho = \widehat{g}(\alpha(\rho),1-\rho)$ where, for any $\alpha \in [0,\alpha_c]$, the function $s \mapsto \widehat{g}(\alpha,s)$ is the size-biased version of $s \mapsto{g}(\alpha,s)$ defined in Proposition \ref{prop:law}, namely $\widehat{g}(\alpha,s) = \partial_s g(\alpha,s) / \partial_s g(\alpha,1)$.
The graph $(t,h(t))_{t \in [0, 2 \xi_{\bm \pi}]}$ can be divided into a first increasing part and a second decreasing part.
These parts are respectively parametrized for $\rho \in [0, \rho_{\bm \pi}]$ by :
\[
\begin{cases}
x^\uparrow(\rho) & := (2-\rho) \, \alpha(\rho) - \int_{\rho}^{\rho_{\bm \pi}} \alpha(u) \mathrm{d}u, \\
y^\uparrow(\rho) & := \rho \, \alpha(\rho) + \int_{\rho}^{\rho_{\bm \pi}} \alpha(u) \mathrm{d}u,
\end{cases}
\]
for the increasing part and
\[
\begin{cases}
x^\downarrow(\rho) := x^\uparrow (\rho) + 2 \, \left(1- \alpha(\rho) \right) \bigg( 1 -  g \big(\alpha(\rho),1-\rho \big) \bigg),\\
y^\downarrow(\rho) := y^\uparrow(\rho),
\end{cases}
\]
for the decreasing part.
\end{theorem}
A direct consequence of this result in the following.
\begin{corollary} \label{coro:length}
Let $\mathcal{H}_N$ be the length of the longest simple path in a configuration model of size $N$ with asymptotic distribution $\bm \pi$ satisfying hypothesis of Proposition \ref{prop:law}. Then, with the notation of Theorem \ref{th:profile},
\[  \forall \varepsilon>0, \quad \mathbf{P} \left( \frac{\mathcal{H}_N}{N}  \geq y^\uparrow(0) - \varepsilon = \int_0^ {\rho_{\bm \pi}} \alpha(u) \mathrm{d}u - \varepsilon \right) \underset{ N \rightarrow +\infty}{ \longrightarrow} 1. \]
\end{corollary}

\begin{remark}
Note that the formulas in Proposition \ref{prop:law} and Theorem \ref{th:profile} have a meaning when $\bm \pi$ has a first moment. Therefore, it is natural to expect that the restriction on the tail of $\bm \pi$ is only technical.
\end{remark}

\begin{remark}
Theorem~\ref{th:profile} and Corollary~\ref{coro:length} are still valid when we condition the graphs to be simple, in which case the configuration model is the uniform random graph with a given degree sequence. This is a classical consequence of the fact that the probability to be simple for a configuration model with a given asymptotic degree distribution is uniformly bounded away from zero under our assumptions. See for instance van der Hofstad's book \cite{vdH}.
\end{remark}

\section{Examples}\label{sec:examples}
In this section we provide explicit formulations of Proposition \ref{prop:law} and Theorem \ref{th:profile} for particular choices of the initial probability distribution $\bm \pi$.

\subsection{Poisson distribution}
Since the Erd\H{o}s-Rényi model on $N$ vertices with probability of connection $c/N$ is contiguous to the configuration model on $N$ vertices with sequence of degree $D_1^{(N)}, \ldots, D_N^{(N)}$ that are i.i.d. with Poisson law of parameter $c$, we can recover the result of Enriquez, Faraud and Ménard \cite{EnriquezFaraudMenard20}. Indeed, in the Erd\H{o}s-Rényi case, after having explored a proportion $\alpha$ of vertices, the graph induced by the unexplored vertices is an Erd\H{o}s-Rényi random graph with $(1-\alpha)N$ vertices and parameter $c/N$, hence its asymptotic degree distribution is Poisson with parameter $(1-\alpha)c$. This is in accordance with our Proposition \ref{prop:law} since in that case, denoting $f(s)=\exp( c(s-1))$ the generating series of the Poisson law with parameter $c$,
\begin{align*}
g(\alpha,s) 
&= \frac{1}{1-\alpha} f \left(  f^{-1}(1-\alpha) - (1-s) \frac{f'(f^{-1}(1-\alpha))}{f'(1)} \right) \\
&= \frac{1}{1-\alpha} \exp \left(  c\left( f^{-1}(1-\alpha) - (1-s) \frac{f'(f^{-1}(1-\alpha))}{f'(1)}  -  1 \right)    \right) \\
&= \frac{1}{1-\alpha} \exp \left(  c\left( 1 + \frac{\log(1-\alpha)}{c} - (1-s)  \frac{cf(f^{-1}(1-\alpha))}{c}   -  1 \right)    \right) \\
&= \frac{1}{1-\alpha} \exp \left(  c\left( 1 + \frac{\log(1-\alpha)}{c} - (1-s)  (1-\alpha)   -  1 \right)    \right) \\
&= \exp \left( c (1-\alpha)(s-1)  \right).
\end{align*}
Using the formulas of Theorem \ref{th:profile}, we obtain the same equations as in \cite{EnriquezFaraudMenard20} for the limiting profile of the DFS spanning tree.

\begin{figure}[ht!]
\centering
\begin{tabular}{ccc}
\includegraphics[scale=0.22]{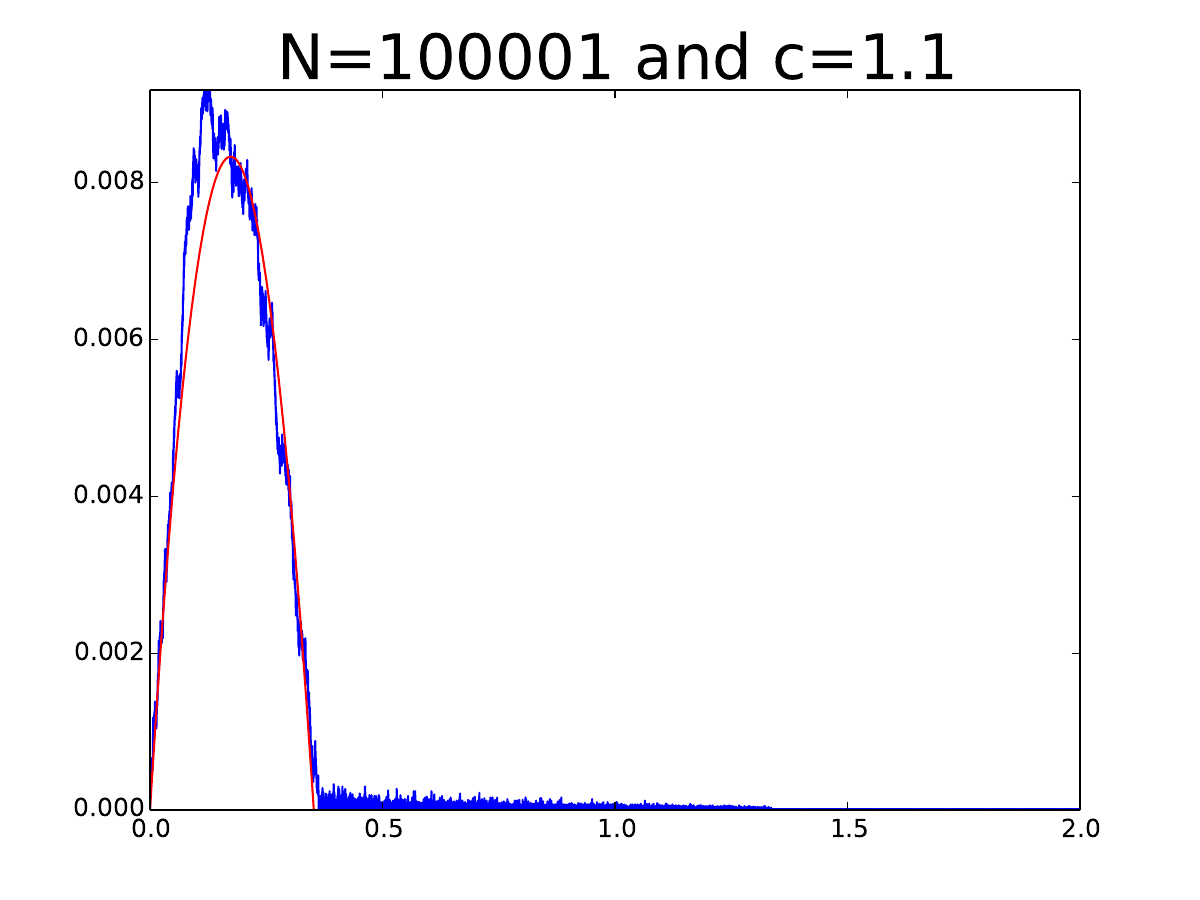}
&
\includegraphics[scale=0.22]{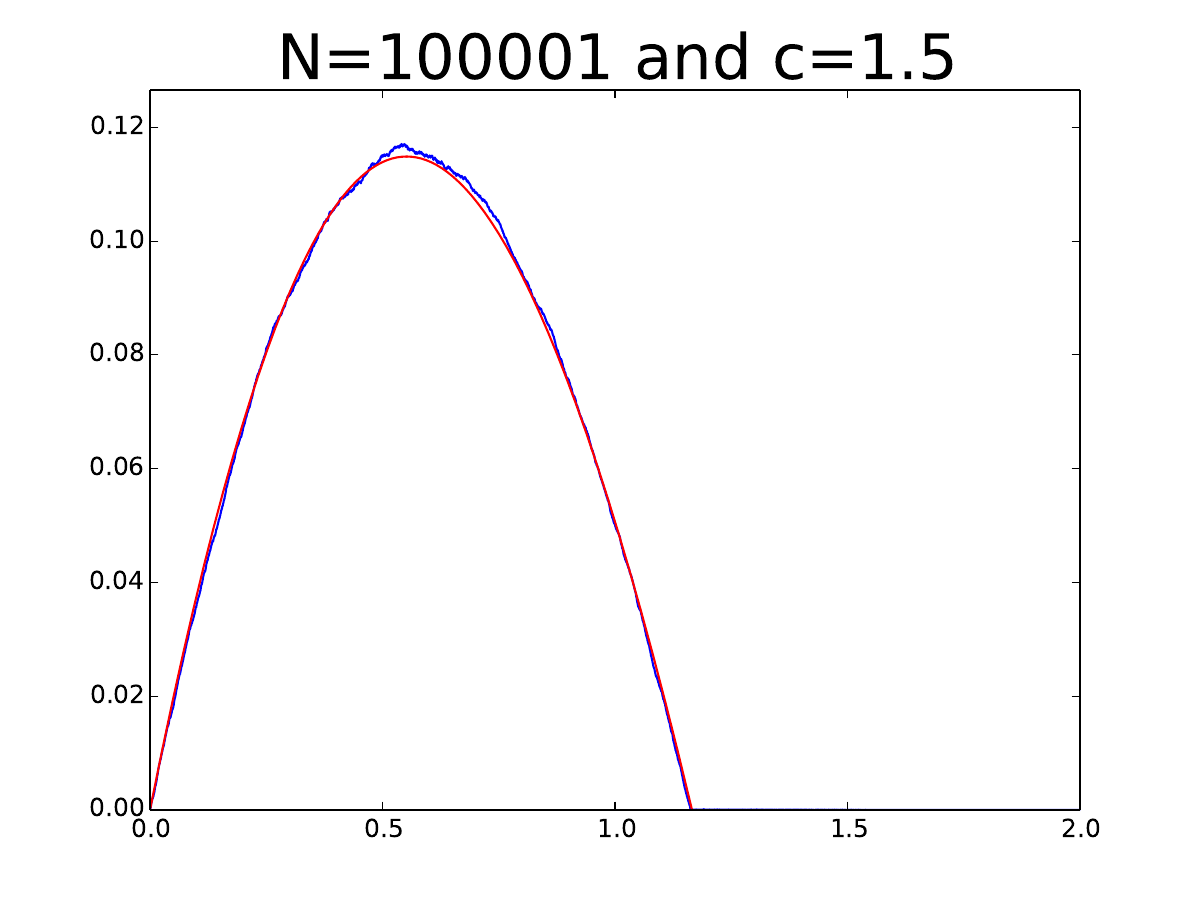}
&
\includegraphics[scale=0.22]{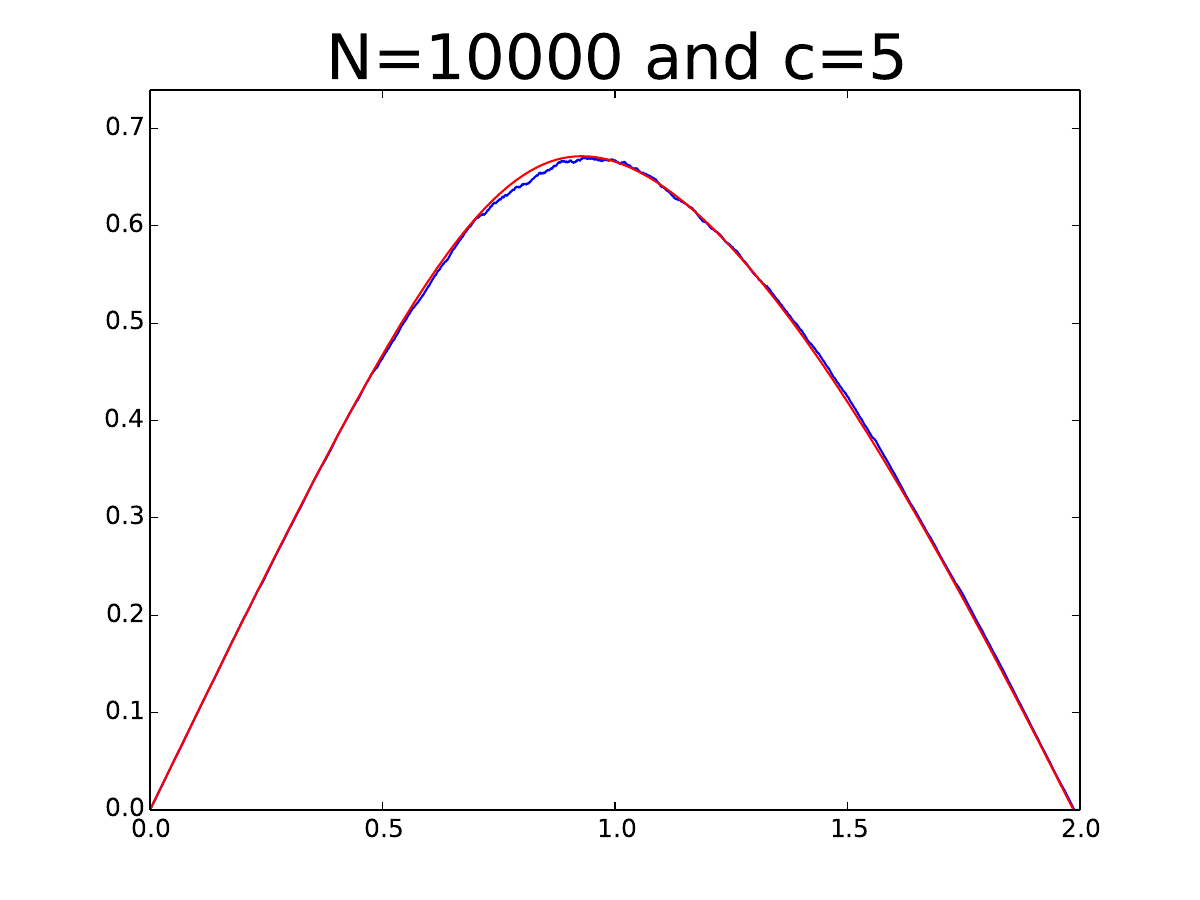}
\end{tabular}
\caption{\label{fig:Poisson}Simulations of $(X_{\lceil tN\rceil }/N)_{t\in [0,2]}$ (blue) and the limiting shape (red) for various values of $N$ and $c$. Notice that when $c$ is close to $1$, we have to take $N$ very large for the walk to be close to its limit.}
\end{figure}

\subsection{$d$-Regular and Binomials distributions}
Let $d \geq 3$. Since the results of Proposition \ref{prop:law} and Theorem \ref{th:profile} hold with probability tending to $1$, we can obtain results on $d$-regular uniform random graphs by applying them to the contiguous model which consists in choosing $\bm \pi = \delta_d$ and conditioning the graphs to be simple. By Proposition \ref{prop:law}, the degree distribution $\bm \pi_\alpha$ has generating function
\begin{align}
g(\alpha,s) 
&= \frac{1}{1-\alpha}  \left(  (1-\alpha)^{1/d} - (1-s) \frac{d(1-\alpha)^{(d-1)/d}}{d} \right)^d \nonumber \\
&= \left( 1 + (s-1)(1-\alpha)^{ \frac{d-2}{d} }  \right)^d \label{eq:gendreg}.
\end{align}
Hence, $\bm \pi_\alpha$ is a binomial distribution $\mathrm{Bin}\left(d, (1-\alpha)^{ \frac{d-2}{d} }  \right)$. From \eqref{eq:gendreg}, we get $\hat{g}(\alpha,s)= (1+(s-1)(1-\alpha)^{(d-2)/d})^{d-1}$. Solving the equation $1-\rho = \hat{g}(\alpha,1-\rho)$ in $\alpha$ gives:
\[  \alpha(\rho) = 1 - \left( \frac{1 - (1-\rho)^{ \frac{1}{d-1} }}{\rho}  \right)^{ \frac{d}{d-2} }.  \]
From this, we deduce a parametrization of the limiting profile in terms of hypergeometric functions. In particular, the height of the limiting DFS spanning tree is given by
\[  H_{\max}(d) = 1 -  \int_0^1 \left( \frac{1 - x^{ \frac{1}{d-1} }}{1-x}  \right)^{ \frac{d}{d-2} } \mathrm{d}x. \]

When $\bm \pi$ has a binomial distribution with parameters $d$ and $p$, $\bm \pi_\alpha$ is also a binomial distribution. 
\[ \bm \pi_\alpha = \mathrm{Bin} \left(   d, p ( 1 - \alpha)^{\frac{d-2}{d}} \right). \]

\begin{figure}[ht!]
\centering
\begin{tabular}{ccc}
\includegraphics[scale=0.22]{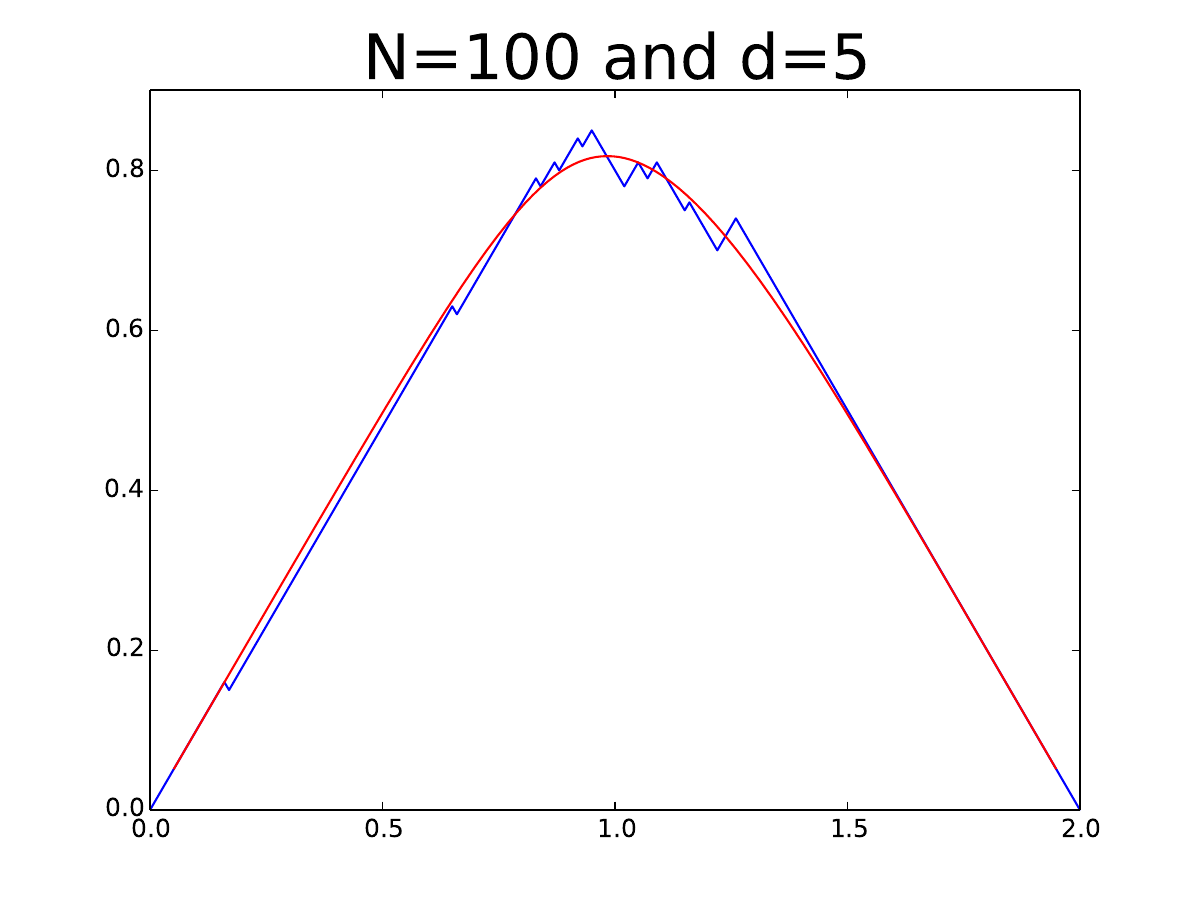}
&
\includegraphics[scale=0.22]{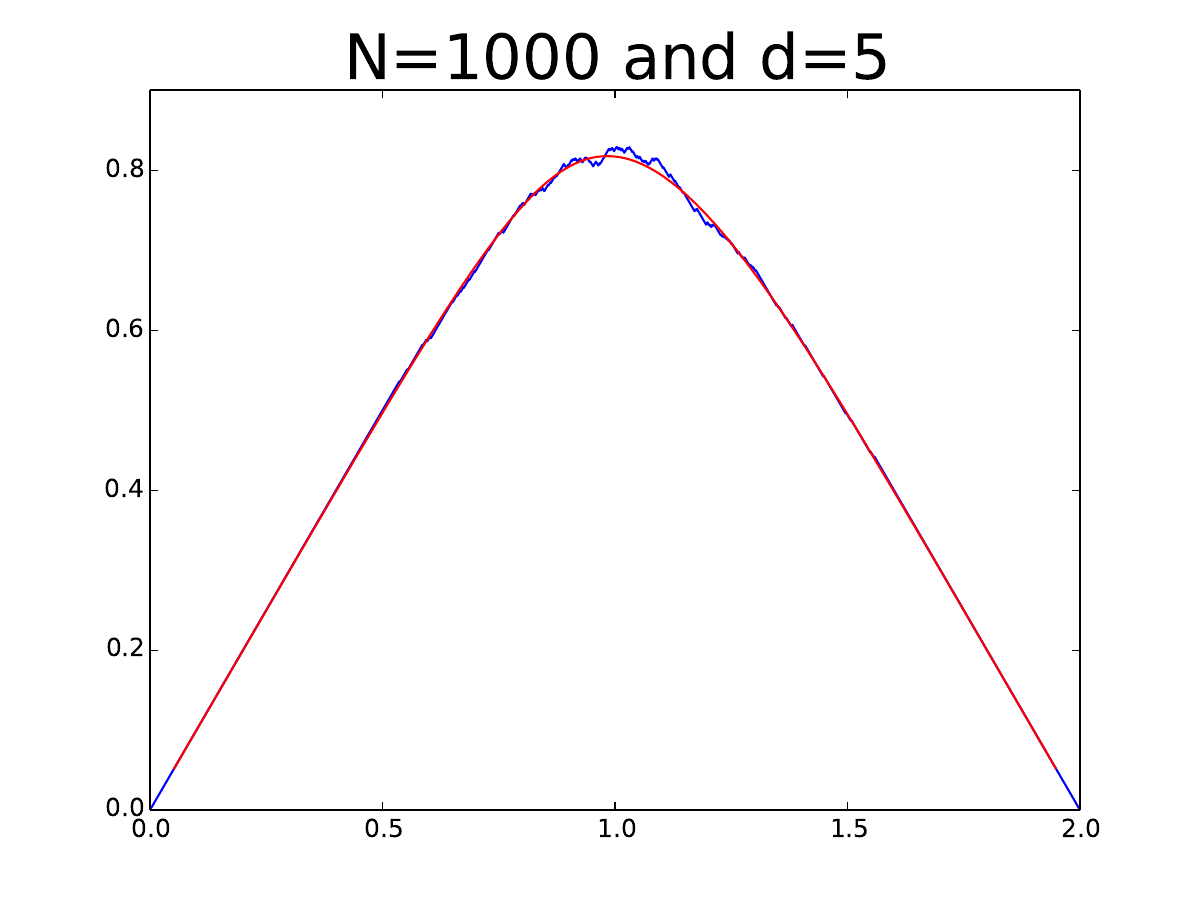}
&
\includegraphics[scale=0.22]{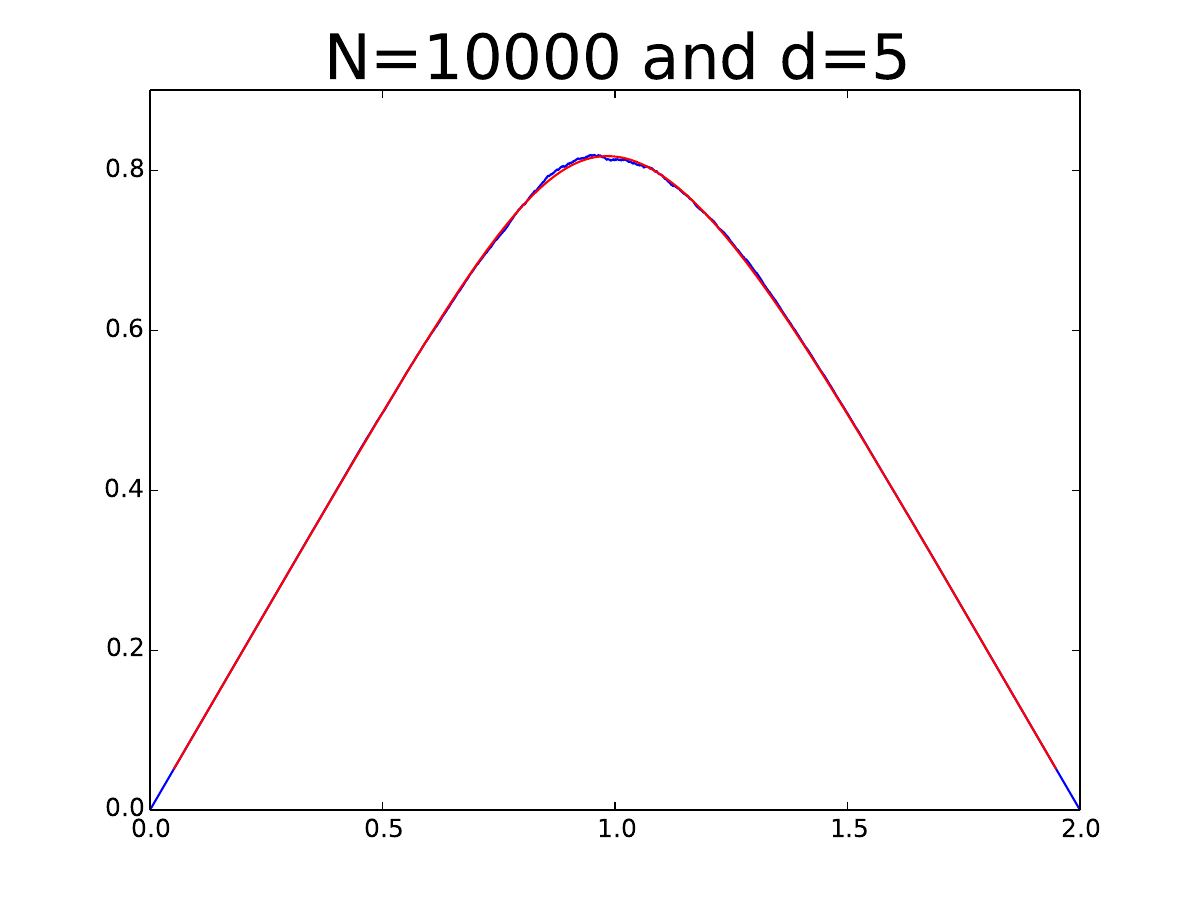}
\end{tabular}
\caption{\label{fig:dreg}Simulations of $(X_{\lceil tN\rceil }/N)_{t\in [0,2]}$ (blue) and the limiting shape (red) for $5$-regular graphs of various sizes.}
\end{figure}

\subsection{Geometric distribution}
Let $p>0$ and suppose that the initial distribution $\bm \pi$ is a geometric distribution starting at $0$ with parameter $p$. The generating series of $\bm \pi$ is $f(s)= \frac{p}{1-(1-p)s}$. We assume $p < 2/3$ so that the configuration model with asymptotic degree distribution $\bm \pi$ has a giant component. Then, by Proposition \ref{prop:law}, the distribution $\bm \pi_\alpha$ has generating series
\[ g(\alpha,s) = \frac{p(\alpha)}{1-(1-p(\alpha))s},  \]
where $p(\alpha)= \frac{p}{p+(1-p)(1-\alpha)^3}$. Hence, $\bm \pi_\alpha$ is a geometric distribution that starts at $0$ with parameter $p(\alpha)$. The generating series of $\hat{\bm \pi}_\alpha$ is $\hat{g}(\alpha,s)= \left(  \frac{p(\alpha)}{1-(1-p(\alpha))s} \right)^2$. Therefore, the solution in $\alpha$ of $1-\rho = \hat{g}(\alpha,1-\rho)$ is
\[  \alpha(\rho) = 1 - \left(\frac{p}{1-p}\right)^{1/3} \left( \frac{1}{1-\rho + \sqrt{1-\rho} }  \right)^{1/3}.  \]
In particular, the height of the limiting DFS spanning tree is given by:
\[  H_{\max}(p) = \rho_{\bm \pi} -  \left(\frac{p}{1-p}\right)^{1/3} \int_0^{\rho_{\bm \pi}} \left(\frac{1}{x + \sqrt{x}} \right)^{1/3} \mathrm{d}x, \]
where $\rho_{\bm \pi}$ is given by:
\[  \rho_{\bm \pi} = \frac{1}{2} \left( \frac{1 - 3p}{1-p} + \sqrt{\frac{1+3p}{1-p}} \right).  \]

\begin{figure}[ht!]
\centering
\begin{tabular}{ccc}
\includegraphics[scale=0.22]{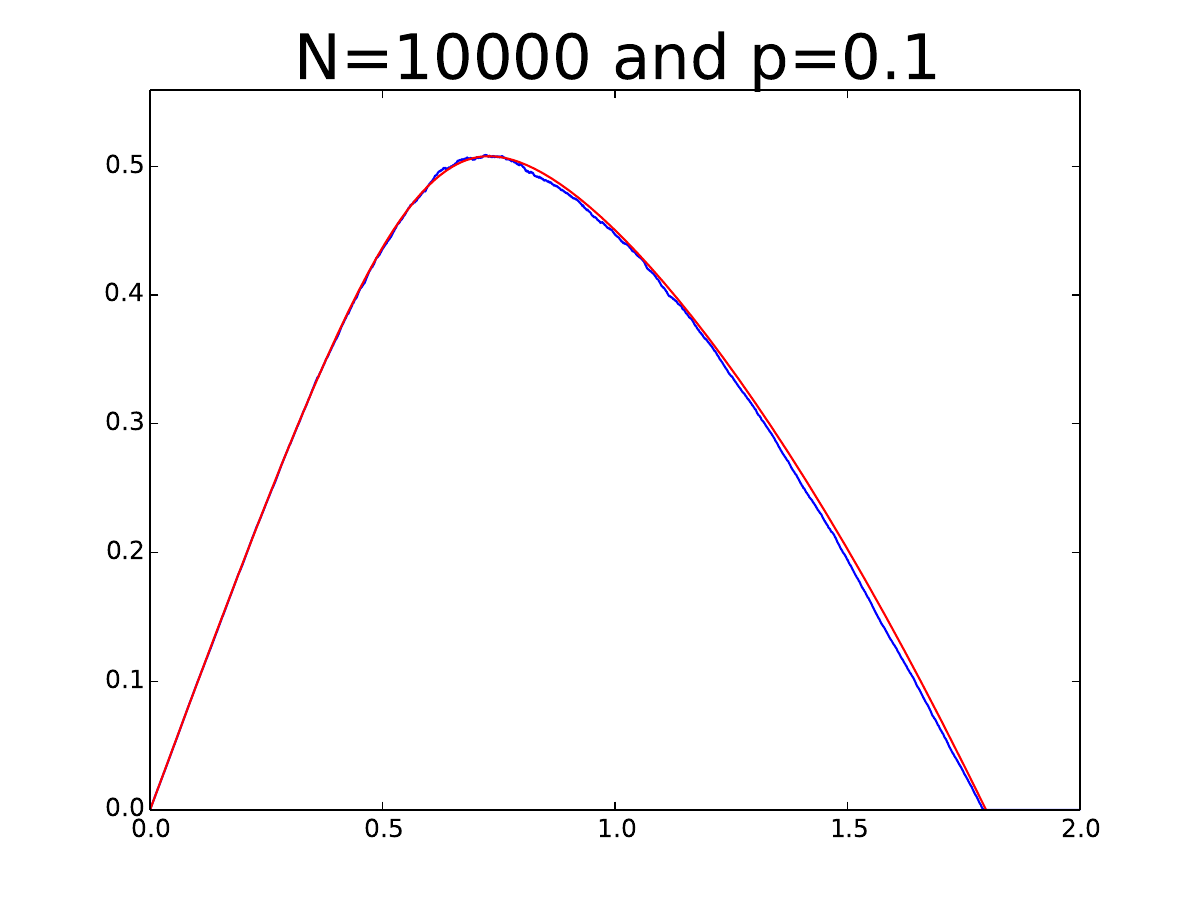}
&
\includegraphics[scale=0.22]{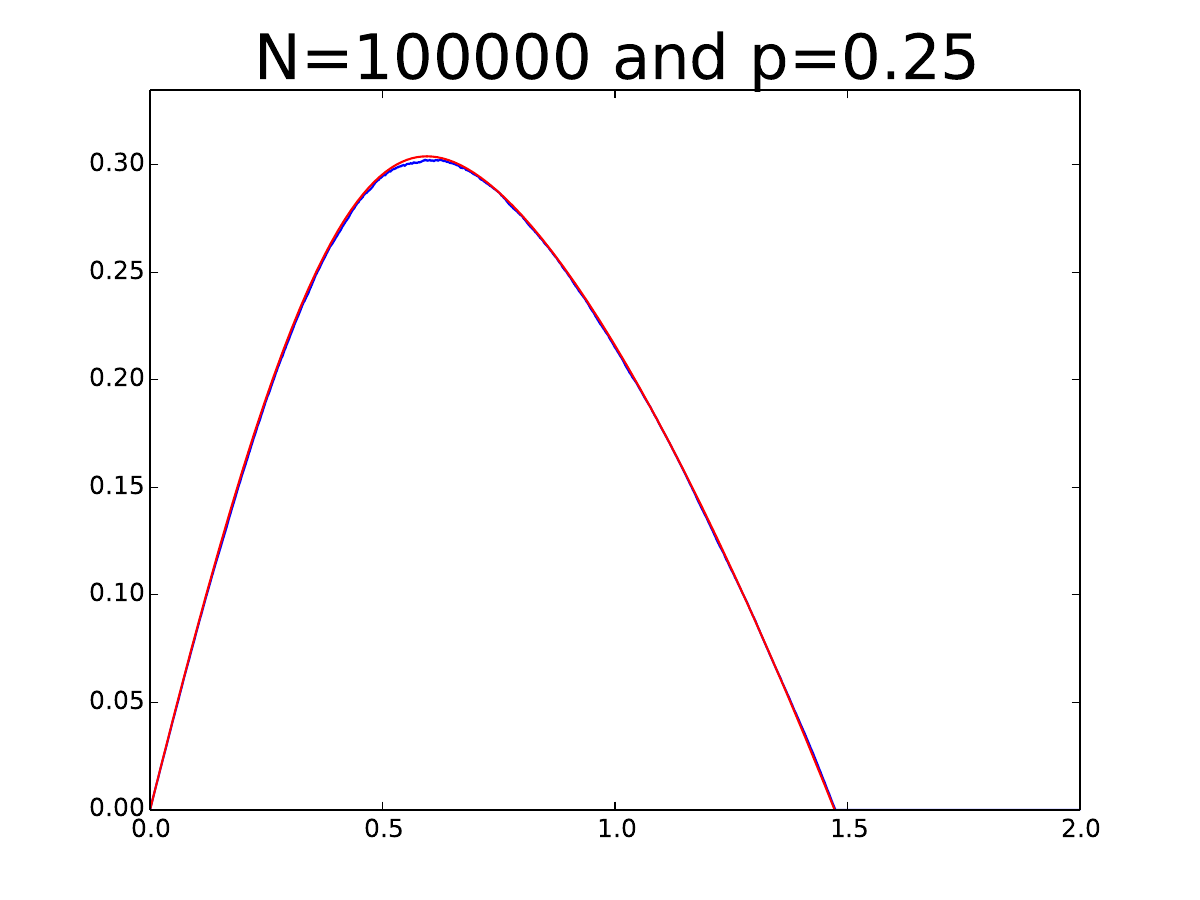}
&
\includegraphics[scale=0.22]{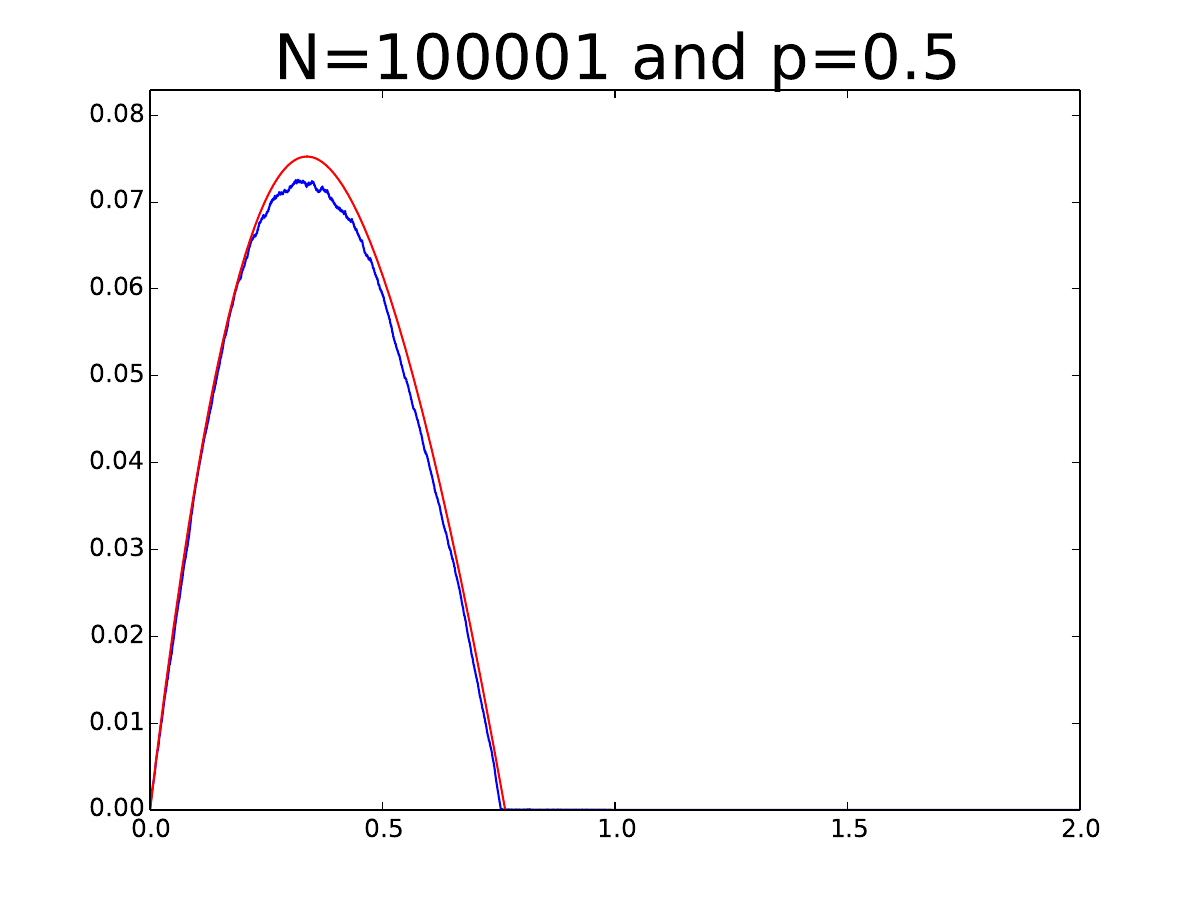}
\end{tabular}
\caption{\label{fig:geom}Simulations of $(X_{\lceil tN\rceil }/N)_{t\in [0,2]}$ (blue) and the limiting shape (red) for random graphs with geometric degrees with various perimeters.}
\end{figure}

\subsection{Heavy tailed distribution}
When $\bm \pi$ is a power law distribution of parameter $\gamma>2$, that is when $\bm \pi(\{k, k+1, \ldots \}) \sim C/k^\gamma$ for a constant $C$, only the first $\lfloor \gamma \rfloor$ moments of $\bm \pi$ are finite. Let $\alpha \in (0,\alpha_c)$. Then, for all $n \geq 0$, the $n$-th factorial moment of $\bm \pi_\alpha$ is equal to
\begin{align*} 
\bm \pi_\alpha(x^n) 
&= \frac{\partial^n}{\partial s^n} \Bigg|_{s=1} g(\alpha,s) \\
&= \left( \frac{f_{\bm \pi}' \left(  f_{\bm \pi}^{-1}(1-\alpha) \right)}{f_{\bm \pi}'(1)} \right)^n  \frac{f^{(n)}\left( f^{-1}(1-\alpha) \right)}{1-\alpha} .
\end{align*}
Therefore, after  visiting a proportion $\varepsilon N$ of the vertices in the DFS, the asymptotic distribution of the degrees of the graph induced by the unexplored vertices is not a power law and has moments of all orders. This remarkable phenomenon could be explained by the fact that vertices of high degree are visited in a microscopic time. We believe that a precise study of this case could be of independent interest.

\section{Constructing while exploring} \label{sec:ConstrDFS}
Let $(\mathbf d^{(N)})_{N\geq 1}$ be a sequence of degree sequences of increasing length satisfying the assumptions of Proposition \ref{prop:law}. For a fixed $N \geq 1$, we use the sequence $\mathbf d^{(N)} = (d_1^{(N)}, \ldots, d_N^{(N)})$ to construct a configuration model $\mathscr{C}(\mathbf{d}^{(N)})$ with vertex set $\mathrm{V}_N= \{1, \ldots,N\}$. More precisely, we simultaneously build the graph and its DFS exploration. This will be done in a similar way as for the DFS defined in Section \ref{sec:defnDFS}, while revealing as little information about the unexplored part of the graph as possible. For every step $n$ we consider the following objects, defined by induction.
\begin{itemize}
\item $A_n$, the active vertices, is an ordered list of pairs $(v,\mathbf m_v)$ where $v$ is a vertex of $\mathrm{V}_N$ and $\mathbf m_v$ is the list of vertices corresponding to the vertices that will be \textbf{m}atched to $v$ during the rest of the exploration.
\item $S_n$, the sleeping vertices, is a subset of $\mathrm{V}_N$. This subset will never contain a vertex of $A_n$.
\item $R_n$, the retired vertices, is another subset of $\mathrm{V}_N$ composed of all the vertices that are neither in $A_n$ nor $S_n$.
\end{itemize}
At time $n=0$, choose a vertex $v$ uniformly at random and pair each of its $d_v^{(N)}$ half edges to a half edge of the graph. This gives an unordered set of vertices that will be matched to $v$ at some point of the exploration. We denote by $\mathbf m_v$ this set with a uniform order. Set:
\[
\begin{cases}
A_0 &= \left((v,\mathbf m_v) \right), \\
S_0 &= \mathrm{V}_N \setminus \{v\}, \\
R_0 &= \emptyset.
\end{cases}
\]
Suppose that $A_n$, $S_n$ and $R_n$ have already been constructed. Three cases are possible.
\begin{enumerate}
\item If $A_n = \emptyset$, the algorithm has just finished exploring and building a connected component of $\mathscr{C}( \bm d^{(N)} )$. In that case, we pick a vertex $v_{n+1}$ uniformly at random from $S_n$ and we pair each of its $d_{v_{n+1}}^{(N)}$ half edges to a uniform half edge belonging to a vertex of $S_n$. We denote by $\mathbf m_{v_{n+1}}$ the set of these paired vertices which are different from $v_{n+1}$ (corresponding to loops in the graph), ordered uniformly and set:
\[
\begin{cases}
A_{n+1} &= \left(v_{n+1},\mathbf m_{v_{n+1}} \right), \\
S_{n+1} &= S_n \setminus \{v_{n+1}\}, \\
R_{n+1} &= R_n.
\end{cases}
\]

\item If $A_n \neq \emptyset$ and if its last element $(u,\mathbf m_u)$ is such that $\mathbf m_u = \emptyset$, the DFS backtracks and we set:
\[
\begin{cases}
A_{n+1} &= A_n - (u,\mathbf m_u), \\
S_{n+1} &= S_n,  \\
R_{n+1} &= R_n \cup \{u\}.
\end{cases}
\]

\item If $A_n \neq \emptyset$ and if its last element $(u,\mathbf m_u)$ is such that $\mathbf m_u \neq \emptyset$, the algorithm goes to the first vertex of $\mathbf m_u$, say $v_{n+1}$. By construction, this vertex always belongs to $S_n$.
We first update $A_n$ into $A'_n$ by removing each occurrence of $v_{n+1}$ in the lists $\mathbf m_x$ for $x \in A_n$.
The half edges of $v_{n+1}$ that have not been matched up to now are uniformly matched with half edges of $S_n$ that have not yet been matched. We order the set of corresponding vertices and denote $\mathbf m_{v_{n+1}}$ this list. We finally set
\[
\begin{cases}
A_{n+1} &= A'_n + (v_{n+1},\mathbf m _{v_{n+1}}) \\
S_{n+1} &= S_n \setminus \{v_{n+1}\},  \\
R_{n+1} &= R_n.
\end{cases}
\]

\end{enumerate}
\bigskip

Since each matching of half-edges in the algorithm is uniform, it indeed constructs a random graph $\mathscr{C}(\bm d^{(N)})$. Moreover, as advertised at the end of Section \ref{sec:defnDFS}, this algorithm simultaneously constructs the DFS on this random graph as each of the three cases are in correspondence to the same three cases in the definition of the DFS given in Section \ref{sec:defnDFS}.

From this construction, it is clear that for every $n$, the graph induced by $S_n$ in the whole graph is a configuration model conditionally on the induced degree sequence. Moreover, for each vertex $v$ of $S_n$, its degree in this induced graph is given by its initial degree $d_v^{(N)}$ minus the number of times that $v$ appears in the lists $\mathbf m_x$ for $x \in A_n$. 

\bigskip
In order to prove Theorem 1, we will first analyse the part of the algorithm corresponding to the increasing part of the limiting profile. This has the same law as the increasing part of the process $(X_n)_{0 \leq n \leq 2N} = (|A_n|)_{0 \leq n \leq 2N}$. During this first phase, at each time, the graph induced by the sleeping vertices, which we will call the remaining graph, is a supercritical configuration model. We will see in Section \ref{sec:pseudo} that there is a sequence of random times where the DFS discovers a vertex belonging to what will turn out to be the giant component of the remaining graph. We will call these times ladder times and study in detail the law of the remaining graph at these times in Section \ref{sec:graphs}.

\subsection{Ladder times} \label{sec:pseudo}
Fix $\delta \in (0,1) $. Let $T_0 = 0$ and define, for $k \in \{0 , \ldots , K\}$,
\begin{equation*}
T_{k+1} := \min \left\{ i > T_k, \, X_i = k+1 \, \, \text{and} \, \, \forall i \leq j \leq i + N^\delta, \, X_j \geq k+1 \right\},
\end{equation*}
where $K$ is the last index for which this definition makes sense (i.e. the set for which the min is taken is not empty).  Of course, this sequence of times will only be useful to analyze the DFS on $\mathscr C (\mathbf d^{(N)})$ when $K$ is of macroscopic order, which is indeed the case with high probability under the assumptions of Proposition \ref{prop:law}.

For all $k \in \{0 , \ldots , K\}$, let $\mathscr{S}_k$ be the graph induced by the vertices of $S_{T_k-1}$ in the graph constructed by the algorithm of the previous section, with the convention that $S_{-1} = \emptyset$. We also denote by $v_k$ the last vertex of $A_{T_k}$. The graphs $\mathscr S_k$ and $S_{T_k}$ have the same vertex set except for $v_k$ which belongs to $\mathscr S_k$ but not to $S_{T_k}$. See Figure \ref{fig:ladder} for an illustration of these definitions. We chose to emphasize $\mathscr{S}_k$ because the structural changes between two such consecutive graphs will be easier to track.
\begin{figure}[!h]
\centering
\includegraphics[scale=0.8]{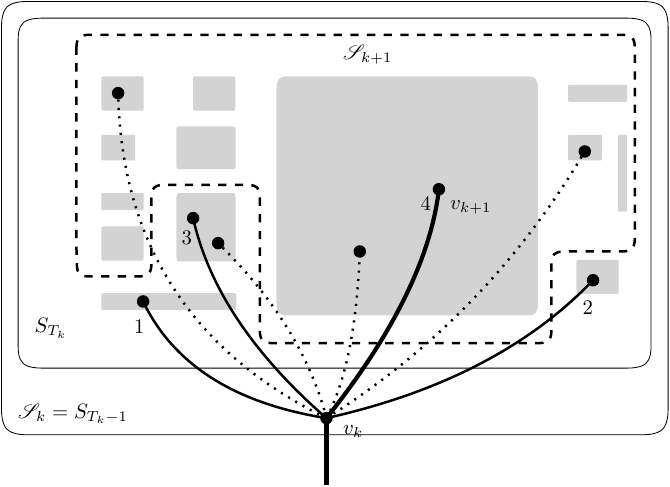}
\caption{\label{fig:ladder} Structure of the remaining graph at a ladder time. The first half edges of $v_k$ are numbered according to their matching order during the construction. Here, the last matched half edge is in bold and connects $v_k$ to $v_{k+1}$. The remaing half edges of $v_k$ are represented by dotted lines and matched to unexplored vertices.}
\end{figure}

Fix $k < K$. From the definition of the times $T_k$ and $T_{k+1}$, we can deduce that $v_{k+1}$ and $v_k$ are neighbors in $\mathscr S_k$. Between the times $n = T_k$ and $n = T_{k+1}$ the process $X_n = |A_n|$ stays above $k$ and is equal to $k$ at time $T_{k+1} - 1$. Each excursion of $X_n$ strictly above $k$ between $T_k$ and $T_{k+1} - 1$ corresponds to the exploration of a different connected component of $\mathscr{S}_k \setminus \{v_k \}$ and we have
\begin{equation*} 
T_{k+1} - T_k = 1 + 2 \times (\text{ number of vertices in ${\mathscr{S}_k} \setminus \mathscr{S}_{k+1}$ } -1).
\end{equation*}
In addition, the definition of the ladder times implies that these connected components have sizes smaller than $N^\delta$.

\bigskip

For every $n \in \{0 , \ldots, 2N\}$, let $D_{n}^{(N)}$ be the degree of a uniform vertex in the graph induced by $S_{n}$.
For every $\varepsilon >0$, we define
\begin{equation*}
n_\epsilon = n_{\varepsilon}^{(N)} = \sup \left\{ n \in \llbracket 0, 2N \rrbracket: \, \forall m \in \llbracket 0,n \rrbracket, \, \frac{\E[D_{n}^{(N)}( D_{n}^{(N)}-1)]}{\E[D_{n}^{(N)}]} > 1 + \varepsilon \right\}.
\end{equation*}
For $n < n_\varepsilon$, the subgraphs induced by $S_n$ are all supercritical. For $0 < \delta < 1/2$, let $\mathbf G_\varepsilon = \mathbf{G}_{\varepsilon}^{(N)}(\delta)$ be the event that, for all $n < n_\varepsilon$,
\begin{itemize}
%\item the maximum degree of a vertex in $S_0$, hence in the graph induced by $S_n$, is at most $N^{1/\gamma}$;
\item there is at least one connected component with size greater than $N^{1-\delta}$ in the graph induced by $S_n$;
\item there is no connected component of size between $N^\delta$ and $N^{1 - \delta}$ in the graph induced by $S_n$.
\end{itemize}
Under the assumptions of Proposition \ref{prop:law} we have, for every $\lambda >0$,
\begin{equation}
\label{lemma: estimate good event}
\mathbb P \big( \mathbf{G}_{\varepsilon} \big) = 1 -  \mathcal O (N^{-\lambda}).
\end{equation}
See for example the last bound page 82 of Bordenave \cite{BordenaveCours}.

\bigskip

The event $\mathbf G_\varepsilon$ will be instrumental in the analysis of the DFS and the times $T_k$ because, on this event, if $T_k < n_\varepsilon$, then the graph $S_{T_k} = {\mathscr S}_k \setminus \{v_k\}$ has a connected component of size larger than $N^{1-\delta}$ and, in $\mathscr S_k$, the vertex $v_k$ has a neighbor in this giant component. Indeed, if every neighbor of $v_k$ in $S_{T_k}$ belonged to a small component, the size of the connected component of $v_k$ in $\mathscr S_k$ would be at most $N^{1/\gamma} \, N^\delta \ll N^{1-\delta}$. On the other hand, we know that this component has size larger than $N^\delta$ meaning that, on $\mathbf G_\varepsilon$, it is in fact larger than $N^{1-\delta}$ leading to a contradiction.
By induction, this means that on $\mathbf G_\varepsilon$ and if $T_k < n_\varepsilon$, then $k<K$.

Let us finally set
\[ K_\varepsilon := \sup \{ k \in \llbracket 0, K \rrbracket, \, T_k < n_\varepsilon   \}, \]
and note that, thanks to \eqref{lemma: estimate good event}, $K_\varepsilon < K$ with probability $1 - \mathcal{O}(N^{-\lambda})$.

\subsection{Analysis of the graphs $\mathscr S_k$} \label{sec:graphs}

Let $N_i(k)$ be the number of vertices of degree $i$ in $\mathscr{S}_k$. The graph $\mathscr{S}_k$ has the law of a configuration model with vertex degrees given by the sequence $(N_i(k))_{i \geq 0}$. Denote by $V_i(S)$ the number of vertices with degree $i$ in the graph $S$. Moreover if $H$ is a subgraph of $S$, $S\setminus H$ stands for the subgraph of $S$ induced by its vertices that do not belong to $H$. Recalling that $\mathbf m_{v_k}$ denotes the list of neighbors of $v_k$ in $\mathscr{S}_k$ (self-loops not included), the evolution of $N_i$ is given by:
\begin{align}
N_i(k+1)-N_i(k) = 
& - V_i({\mathscr{S}_k} \setminus \mathscr{S}_{k+1}) \label{eq2} \\
& + \sum_{v \,\in  \, \mathbf m_{v_k} \cap \mathscr S_{k+1}}
\left( - \mathbf{1}_{ \deg_{\mathscr{S}_k}(v) =i } + \mathbf{1}_{ \deg_{\mathscr{S}_k}(v)= i+ o_v } \right) \label{eq3},                
\end{align}
where $o_v$ is the number of occurrences of $v$ in $\mathbf{m}_{v_k}$. Indeed, the first contribution corresponds to the complete removal of vertices belonging to $\mathscr S_k$ but not to $\mathscr S_{k+1}$. The second contribution corresponds to  edges of $\mathscr{S}_k$ connecting $v_k$ to vertices of $\mathscr{S}_{k+1}$, taking into account eventual multiple edges. Figure \ref{fig:ladder} gives an illustration of this situation. In this figure, the contribution \eqref{eq2} comes from the connected components of the vertices attaches to the half edges of $v_k$ numbered $1$, $2$ and $3$. The contribution \eqref{eq3} comes from $v_{k+1}$ and the vertices matched to dotted half edges.

\bigskip

A fundamental step in understanding the behaviour of the exploration process is to identify the asymptotic behaviour of the variables $T_k$ and $N_i(k)$ for large $N$. This is the object of Theorem \ref{theo: fluidlimits}. To state this, we first introduce some technical notation.

Let $(z_i)_{i \geq 0} \in \mathbb{R}^{ \mathbb{Z}_+ }$ be such that $\sum_{ i \geq 0} z_i  \leq 1$ and $\sum_{k \geq 0} i z_i < \infty$. 
For any $i \geq 0 $ let $\hat{z}_i = (i+1)z_i / \sum_j j z_j$ and define:
\begin{equation}
\begin{cases} g_{(z_i)_{i \geq 0}}(s) &=  \sum\limits_{i \geq 0} \frac{z_i}{\sum_{l\geq 0} z_l} s^i \\
\hat{g}_{ (z_i)_{i \geq 0}}(s) &= \sum\limits_{i \geq 0} \hat{z}_i s^i = \frac{g'_{(z_i)_{i \geq 0} }(s)}{g'_{(z_i)_{i \geq 0} } (1)}
\end{cases}
\label{eq: DefnGenSer}
\end{equation}
respectively the generating series associated to $(z_k)_{k \geq 0}$ and its sized-biased version.
Let also $\rho_{(z_i)_{i\geq 0}}$ be the largest solution in $[0,1]$ of 
\begin{equation} \label{eq:rho}
1-s = \hat{g}_{(z_i)_{i \geq 0} } (1-s).
\end{equation}
\begin{remark}
Since $\hat{g}$ is the generating function of a probability distribution on the integers, it is convex on $[0,1]$. Therefore, Equation \eqref{eq:rho} has a positive solution in $(0,1]$ if and only if $\hat{g}'(1) > 1$, which is equivalent to $\frac{\sum_{l \geq 1} (l-1)l z_l}{\sum_{l \geq 1} l z_l} > 1$.
\end{remark}

We also define the following functions:
\begin{align}
f(z_0,z_1, \ldots) &= \frac{2- \rho_{(z_i)_{i \geq 0}} }{\rho_{(z_i)_{i \geq 0}}} \label{eq:deff} \\
f_i(z_0,z_1, \ldots) &= - \frac{1}{\rho_{(z_j)_{j \geq 0}}} \frac{i z_{i}}{\sum_{j \geq 0} j z_j} \notag \\
& \quad + \frac{1}{\rho_{(z_j)_{j \geq 0}}} \left( 1 - \frac{\sum_{j \geq 0} (j-1)j z_j}{ \sum_{n \geq 0} j z_j }  \right) \left( \frac{i z_{i}}{\sum_{j \geq 0} j z_j} - \frac{(i+1)z_{i+1}}{\sum_{j \geq 0} j z_j} \right) \label{eq:deff_i}.
\end{align}

The asymptotic behaviour of the variables $T_k$ and $N_i(k)$ will be driven by the solution of an infinite system of differential equations whose existence is provided by the following lemma, whose proof is postponed to Section \ref{sec: EqDiff}. Actually, we exhibit an explicit solution of another infinite system of differential equations $(S')$ which is related to the following system \eqref{eq: EqDiff} by some time change.

\begin{lemma}\label{lemma: EqDiff}
Let $\bm \pi = (\bm \pi_i)_{i \geq 0} \in [0,1]^\mathbb{N}$ such that $\sum_{i \geq 0} \bm \pi_i = 1$. Then, the following system of differential equations:
\begin{equation}\tag{S}\label{eq: EqDiff}
\left\{ 
\begin{array}{lcl}
\frac{\mathrm{d}z_i}{\mathrm{d}t} & = & f_i(z_0, z_1, \ldots ); \\
z_i(0) & = & \bm \pi_i.
\end{array}
\right.
\end{equation}
admits a solution $(z_i^*)_{i \geq 0}$ which is well defined on $[0,t_{\max})$ for some $t_{\max}>0$ and whose derivatives $dz_i^*/dt$ are all Lipschitz.

\end{lemma}
We are now ready to state the main result of this section.
\begin{theorem}\label{theo: fluidlimits}
For all $t \in [0,1]$ such that  $\lfloor t N \rfloor \leq K_\varepsilon$, the following convergences in probability hold: 
\begin{align*}
\forall i \geq 1, \quad \frac{N_i(\lfloor tN \rfloor)}{N} &\overset{\mathbb{P}}{\underset{N \rightarrow +\infty}{\longrightarrow}}  z^*_i \left( t   \right),
\end{align*}
Moreover, 
\begin{align*}
\frac{T_{\lfloor tN \rfloor}}{N} &\overset{\mathbb{P}}{\underset{N \rightarrow +\infty}{\longrightarrow}} z^*\left(t \right).
\end{align*}
\end{theorem}
\begin{remark} By Theorem \ref{theo: fluidlimits}, we deduce that the system $(S)$ has a unique solution among sequences of functions with Lipschitz derivatives.
\end{remark}

The proof of Theorem \ref{theo: fluidlimits} is crucially based on the following Lemma which identifies the trends of the quantities $T_k$ and $N_i(k)$.

\begin{lemma}\label{lemma:trends}
\begin{enumerate}
\item There exists $0 < \beta < 1/2$ such that with high probability for all $k \leq K_\varepsilon$, 
\[ |T_{k+1} - T_k | \leq N^\beta \text{   and for all   } k \geq 0,\, |N_i({k+1}) - N_i(k) | \leq N^\beta. \]
\item We denote by $(\mathcal{F}_k )_{k\geq 0}$ the canonical filtration associated to the sequence $\left( (N_i(k))_{i \geq 0} \right)_{k \geq 0} $. There exists $\lambda> 0$ such that for every $k\leq K_\varepsilon$, 
\begin{align*} 
\E[T_{k+1} - T_k \, | \, \mathcal{F}_k] &= f \left( \frac{N_0(k)}{N} , \frac{N_1(k)}{N} , \ldots  \right) + O \left( N^{-\lambda} \right),  \\
 \E[ N_i({k+1}) - N_i(k) \, | \, \mathcal{F}_k] &= f_i \left( \frac{N_0(k)}{N} , \frac{N_1(k)}{N} , \ldots  \right) + O \left( N^{-\lambda} \right).
\end{align*}
\end{enumerate}
\end{lemma}

In the following, we first focus on the proof of Lemma \ref{lemma:trends} and postpone the proof of Theorem \ref{theo: fluidlimits} at the end of the section.

\begin{proof}[Proof of Lemma \ref{lemma:trends}.]
The first point is a consequence of Equation \eqref{lemma: estimate good event} with $\delta < 1/2 - 1/\gamma$. Indeed on the event $\mathbf G_\varepsilon$ the vertices $v_k$ have degree at most $N^{1/\gamma}$ and therefore $T_{k+1} - T_k \leq 1 + 2 N^{1/\gamma} N^\delta \ll N^\beta$ for some $\beta < 1/2$. Since $|N_i({k+1}) - N_i(k)| \leq (T_{k+1} - T_k)/2$ the second inequality is trivial.

In order to establish the second point, we need to analyse the structure of $\mathscr{S}_k$ and the contributions \eqref{eq2} and \eqref{eq3}. To this end, we will study the random variable $\mathfrak{e}_k$ that counts the number of excursions strictly above $k$ of the walker $(X_n)$ coding the DFS between the times $T_k$ and $T_{k+1} -1$ (in Figure \ref{fig:ladder}, $\mathfrak{e}_k = 3$). In particular, the expectation of $\mathfrak{e}_k $ conditionally on $\mathcal F_k$ is well defined on the event $\mathbf G_\varepsilon$.

If we disconnect the edges joining the $\geom_k$ first children of $v_k$ in the tree constructed by the DFS, the remaining connected components in $\mathscr S_k$ of these children have size smaller than $N^\delta$. 
This motivates the following notation:
\begin{itemize}
\item for every $i \geq 0$, let $\mathbf{Ext}_{i}^k$ (resp. $\mathbf{Surv}_{i}^k$) be the set of half-edges $e \in \mathscr{S}_k$ connected to a vertex $w$ of degree $i$ (in $\mathscr{S}_k$) such that the connected component of $w$ after removing this half-edge has size smaller than $N^\delta$ (resp. larger than $N^\delta$); 
\item let $\mathbf{Ext}^k$ (resp. $\mathbf{Surv}^k$) be the set of half-edges $e \in \mathscr{S}_k$ connected to a vertex $w$ such that the connected component of $w$ after removing this half-edge has size smaller than $N^\delta$ (resp. larger than $N^\delta$). Note that $\mathbf{Ext}^k = \sqcup_{j \geq 0} \mathbf{Ext}_{j}^k$ and $\mathbf{Surv}^k = \sqcup_{j \geq 0} \mathbf{Surv}_{j}^k$.
\end{itemize}
Recall that on $\mathbf{G}_\varepsilon$, for all $k \leq K_\varepsilon$, $v_k$ has a neighbor in $\mathscr{S}_k$ that belongs to a connected component of $\mathscr{S}_k$ with more than $N^{\delta}$ vertices. This means that for every such $k$, with probability $1- \mathcal O(N^{-1-\lambda})$, the random variable $\geom_k$ is the number of half edges of $\mathbf{Ext}^k$ attached to $v_k$ before attaching a half edge of $\mathbf{Surv}^k$ during the DFS. In order to compute its expectation, we first condition on $\{ \mathrm{deg}_{\mathscr S_k} (v_k) = d \}$, with $d>0$ fixed.

Notice that, conditional on the event $\{ \mathrm{deg}_{\mathscr S_k}(v_k) = d \} \cap \{ \geom_k < \mathrm{deg}_{\mathscr S_k}(v_k)  \}$, the law of $(\mathscr S_k,v_k)$ is the law of a rooted configuration model $\mathbf C_{ \mathbf{N}(k)}^d$ with root degree $d$ and degree sequence ${\mathbf{N}}(k) := ( N_i(k) )_{i \geq 0}$, conditioned on the root having one of its half edge paired to an element of $\mathbf{Surv} (\mathbf C_{{\mathbf{N}}(k)}^d)$. We define the new random variable $\tilde \geom _k$ as the number of half edges of the root paired to an element of $\mathbf{Ext}(\mathbf C_{{\mathbf{N}}(k)}^d)$ before pairing a half edge to an element of $\mathbf{Surv}(\mathbf C_{{\mathbf{N}}(k)}^d)$ when doing successive uniform matchings in the configuration model (with the convention $\tilde \geom _k=d$ if the root has no half-edged paired to an element of $\mathbf{Surv}(\mathbf C_{{\mathbf{N}}(k)}^d)$). We have the following equality for all $j$:
\[
\PP \big(
\geom_k = j
\, \big| \, 
\mathcal{F}_k \, \text{and} \, \deg_{\mathscr S_k}(v_k)=d
\big) 
= \PP 
\big( \tilde \geom_k = j \, \big| \, \tilde \geom_k < d \big) + \mathcal O(N^{-1-\lambda} ).
\]

Let 
\[  \tilde{\rho}_k := \frac{| \mathbf{Surv}(\mathbf C_{{\mathbf{N}}(k)}^d)|}{2 |E(\mathbf C_{{\mathbf{N}}(k)}^d)|} = 1 - \frac{| \mathbf{Ext}(\mathbf C_{{\mathbf{N}}(k)}^d)|}{2 |E(\mathbf C_{{\mathbf{N}}(k)}^d)|},   \]
the proportion of half-edges in $\mathbf{Surv}(\mathbf C_{{\mathbf{N}}(k)}^d)$ (resp. $\mathbf{Ext}(\mathbf C_{{\mathbf{N}}(k)}^d)$). This proportion is close to a constant $\rho_k$ that we now define with the help of additional notation. Recalling \eqref{eq: DefnGenSer}, let
\begin{equation*}
\begin{array}{cccc} 
p_i = p_i(k) = \frac{N_i(k)}{\sum_{ j \geq 0} N_j(k)}, & \quad  & \quad &  g_k = g_{(p_j)_{j\geq 0}}, \\
\phantom{b} \hat{p}_i = \hat{p}_i(k) = \frac{(i+1)p_{i+1}(k)}{\sum_{j \geq 0}jp_j(k)}, & \quad  & \quad & \phantom{l} \hat{g}_k = \hat{g}_{(p_j)_{j\geq 0}}=g_{(\hat{p}_j)_{j\geq 0}} ,
\end{array}
\end{equation*}
and let $\rho_k = \rho_{(p_j(k))_{j \geq 0}}$ be the largest solution in $[0,1]$ of $1-s = \hat{g}_k(1-s)$. We have the following lemma, whose proof is postponed to Section \ref{sub:technical_lemmas}.
\begin{lemma}\label{lemme: Bollobas estimates}
For all $0 \leq k \leq K_\varepsilon$, there exists $\lambda>0$ and $\eta> 0$ such that, conditionally on $\mathcal{F}_k$, uniformly in $k$,
\begin{equation*}  
\begin{cases} \PP\left(  \left| \frac{ |\mathbf{Ext}_i(\mathbf C_{{\mathbf{N}}(k)}^d)|}{2 |E(\mathbf C_{{\mathbf{N}}(k)}^d)| } - \frac{i p_i}{g_k'(1)} (1-\rho_k)^{i-1}  \right| \geq  N^{-\lambda}     \right) = \mathcal O \left( N ^{-1-\lambda}\right), \\
\PP\left(  \left| \frac{ |\mathbf{Surv}_i(\mathbf C_{{\mathbf{N}}(k)}^d)|}{2 |E(\mathbf C_{{\mathbf{N}}(k)}^d)| } - \frac{i p_i}{g_k'(1)}(1-(1-\rho_k)^{i-1})  \right| \geq  N^{-\lambda}   \right) = \mathcal O \left( N ^{-1-\lambda}\right), \\
\PP\left( \left| \frac{|\mathbf{Ext}(\mathbf C_{{\mathbf{N}}(k)}^d)|}{2|E(\mathbf C_{{\mathbf{N}}(k)}^d)|} - (1-\rho_k)     \right| \geq N^{-\lambda}   \right) = \mathcal O \left( N ^{-1-\lambda}\right), \\
\PP\left( \left| \frac{|\mathbf{Surv}(\mathbf C_{{\mathbf{N}}(k)}^d)|}{2|E(\mathbf C_{{\mathbf{N}}(k)}^d)|} - \rho_k   \right| \geq N^{-\lambda}  \right) = \mathcal O \left( N ^{-1-\lambda}\right). 
\end{cases}
\end{equation*}
\end{lemma}

Using this lemma, we obtain:
\begin{align}
\PP \big(   \geom_k = j   \, & \big| \, \mathcal{F}_k \, \text{and} \, \deg_{\mathscr S_k}(v_k)=d
\big) \notag\\
& =  \frac{\PP \big(  \{ \tilde\geom_k = j \} \cap \{ \tilde \geom _k < d\} \}    \cap \{
|\tilde{\rho}_k - \rho_k| \leq \mathcal{O}(N^{-\lambda}) \} \big)}{\PP \big( \tilde \geom_k < d \cap \{  |\tilde{\rho}_k - \rho_k| \leq \mathcal{O}(N^{-\lambda}) \} \big)} + \mathcal O \left( N^{-1-\lambda} \right).
\label{eq: geom=j}
\end{align}
Fix $j<d$. To estimate the probabilities in \eqref{eq: geom=j}, we successively match the half edges $c_1, \ldots , c_{j+1}$ of the root uniformly among the half edges of $\mathbf C_{{\mathbf{N}}(k)}^d$. Notice that if none of these half edges are matched together, this is equivalent to an urn model without replacement. At each of these steps, the proportion of available half edges of $\mathbf{Ext}(\mathbf C_{{\mathbf{N}}(k)}^d)$ diminishes and is therefore between $1- \widetilde{\rho}_k - \frac{d}{2|E(\mathbf C_{{\mathbf{N}}(k)}^d)|}$ and $1- \widetilde \rho_k$. Recalling that $|E(\mathbf C_{{\mathbf{N}}(k)}^d)|$ is uniformly of order $N$, we can write for every $j<d$
\begin{multline*}
\frac{\left( 1-\rho_k - C \frac{d}{N} +\mathcal O \left( N^{-\lambda} \right) \right)^j \, \left(\rho_k + \mathcal O \left( N^{-\lambda} \right) \right) }
{1 - \left( 1-\rho_k - C \frac{d}{N} + \mathcal O \left( N^{-\lambda} \right) \right)^d}
+\mathcal O \left( N^{-1-\lambda} \right) \\
\leq 
\PP \big(   \geom_k = j   \, \big| \, \mathcal{F}_k \, \text{and} \, \deg_{\mathscr S_k}(v_k)=d
\big)  \\
\leq \frac{\left( 1-\rho_k +\mathcal O \left( N^{-\lambda} \right) \right)^j \, \left(\rho_k + C \frac{d}{N}+ \mathcal O \left( N^{-\lambda} \right) \right) }
{1 - \left( 1-\rho_k + \mathcal O \left( N^{-\lambda} \right) \right)^d}
+\mathcal O \left( N^{-1-\lambda} \right)
\end{multline*}
where $C$ is a constant and the error terms $\mathcal O(N^{-\lambda})$ are the same everywhere and uniform in $d$. This easily translates into
\begin{align*}
\PP \big(   \geom_k = j   \, & \big| \, \mathcal{F}_k \, \text{and} \, \deg(v_k)=d
\big) \\ 
& =
\frac{\left( 1-\rho_k \right)^j \rho_k}
{1 - \left( 1-\rho_k \right)^d}
\left( 1+ \mathcal O \left(d^2 N^{-1} + dN^{-\lambda} \right) \right) \mathbf{1}_{\{j<d\}}
+\mathcal O \left( N^{-1-\lambda} \right)
\end{align*}
where, once again, the error terms are uniform. We can now compute the conditional expectation of $\geom_k$:
\begin{multline*}
\mathbb E \left[ \geom_k \middle| \mathcal F_k, \mathrm{deg}_{\mathscr S_k}(v_k) = d \right] \\=
\frac{ 1-{\rho}_k }
{ {\rho}_k
\left( 1 - \left(1 -{\rho}_k \right)^d \right) }
\left( -d {\rho}_k  \left(1-{\rho}_k \right)^{d-1} + 1 -  \left(1-{\rho}_k \right)^d \right) \left( 1+ \mathcal O \left(d^2 N^{-1} + dN^{-\lambda} \right) \right) \\
\hspace{13.2cm} + \mathcal O ( N^{-\lambda} ),
\end{multline*}
where the last error term comes from the fact the $\geom _k$ is smaller that $\mathcal O (N)$ by definition.

To finally compute the expectation of $\mathfrak{e}_k$, we want to sum the above equality with respect to the law of ${\rm deg}_{\mathscr S_k} (v_k)$. By construction, in $\mathscr S_{k-1}$, the vertex $v_k$ is attached to $v_{k-1}$ by a half edge of $\mathbf{Surv}^{k-1}$ chosen uniformly. Therefore, by Lemma \ref{lemme: Bollobas estimates}, the law of the degree of $v_k$ in $\mathscr S_k$ is given by
\[
\PP(\deg_{\mathscr S_k}(v_{k}) = d \, | \, \mathcal{F}_k) = \frac{(d+1)p_{d+1}(k-1)}{\rho_{k-1} g_{k-1}'(1)} \left( 1 - (1- \rho_{k-1})^d \right) (1 + \mathcal O(N^{-\lambda})),
\]
where the error term is uniform in $d$ and $k$. We can replace $k-1$ by $k$ in the above probabilities at the cost of a factor $1 + \mathcal O (N^{-\lambda})$ which is uniform in $k$ and $d$. Indeed, on $\mathbf G_\varepsilon$, the difference between $\mathscr S_{k-1}$ and $\mathscr S_k$ consists of at most $N^{1/\gamma}$ components of size at most $N^\delta$ and we have $p_{d}(k-1) = p_{d}(k) \left(1+\mathcal O (N^{1/\gamma +\delta -1}) \right)$ uniformly in $k$ and $d$. The difference between $\rho_{k-1}$ and $\rho_k$ is then of the same order by a Taylor expansion. Therefore
\begin{equation} 
\PP(\deg_{\mathscr S_k}(v_{k}) = d \, | \, \mathcal{F}_k) = \frac{(d+1)p_{d+1}(k)}{\rho_{k} g_{k}'(1)} \left( 1 - (1- \rho_{k})^d \right) (1 + \mathcal O(N^{-\lambda})),
\label{eq: DegSurv}
\end{equation}
and we get:
\begin{multline*}
\mathbb E \left[ \mathfrak e _k | \mathcal F_k \right] \\= 
\frac{ (1-\rho_k) }{ g'_k(1) \rho_k^2 }
\sum_{d\geq 0}
(d+1)p_{d+1}(k)
\left( -d {\rho}_k  \left(1-{\rho}_k \right)^{d-1} + 1 -  \left(1-{\rho}_k \right)^d \right)
\left( 1+ \mathcal O \left(d^2 N^{-1} + dN^{-\lambda} \right) \right) \\
\hspace{13.2cm} + \mathcal O ( N^{-\lambda} ) \\
= \frac{ (1-\rho_k) }{ g'_k(1) \rho_k^2 } \left( g'_k(1) - \rho_k g^{''}_k(1-\rho_k) - g'_k(1 - \rho_k) \right) +  \mathcal{O} (N^{\frac{1}{\gamma} - 1}) \cdot \mathcal{O} \left(  \sum_{ d \geq 0} d^2 p_d(k)  \right) + \mathcal{O}(N^{- \lambda}).
\end{multline*}
Notice that the error $\mathcal O(N^{-\lambda})$ is uniform in $k$ and $d$. Let us prove that $\sum_{ d \geq 0} d^2 p_d(k)$ is of order $1$. First note that it is of the same order as $\frac{1}{N}\sum_{d \geq 0} d^2 N_d(k)$, where we recall that $N_d(k)$ is the number of vertices of degree $d$ in $\mathscr{S}_k$. Indeed the number of vertices of $\mathscr{S}_k$ is of order $N$. Denoting by $N_{\geq d}(k)$ the number of vertices of degree larger than $d$ in $\mathscr{S}_k$, it holds that $N_{ \geq d}(k) \geq N_{ \geq d}(k+1)$ from the definition of the algorithm. This monotonicity implies that
\[ \frac{1}{N} \sum_{d \geq 0} d^2 N_d(k) \leq \sum_{d \geq 0} d^2  \frac{N_d(0)}{N}, \]
where the right-hand side converges to a finite limit by assumption \eqref{assump: moment}. Therefore
\begin{align}
\mathbb E \left[ \mathfrak e _k | \mathcal F_k \right] &= \frac{ (1-\rho_k) }{ g'_k(1) \rho_k^2 } \left( g'_k(1) - \rho_k g^{''}_n(1-\rho_k) - g'_k(1 - \rho_k) \right) + \mathcal{O}(N^{-\lambda})\nonumber  \\
&= \frac{1-\rho_k}{\rho_k} \left( 1- \hat{g}'_k(1-\rho_k)  \right) + \mathcal{O} (N^{-\lambda}), \label{eq: ExpecGeom}
\end{align}
where we used $1-\rho_k = \hat{g}_k(1-\rho_k) = g'_k(1-\rho_k)/g'_k(1)$.

\bigskip

Now that we know more about the random variable $\mathfrak e_k$, we can study in more depth the time difference between two consecutive ladder times.

With high probability, the first $\mathfrak{e}_k$ neighbours of $v_k$ in the tree constructed by the DFS all belong to distinct connected components of $\mathscr{S} _k \setminus \{ v_k \}$. We denote these components by $W^{(1)}, \ldots, W^{(\mathfrak{e}_k)}$. Notice that by Lemma \ref{lemme: Bollobas estimates}, for all $i \geq 0$, the ratio $|\mathbf{Ext}_{i}^k| / | \mathbf{Ext}^k|$ concentrates around $i p_i(k)(1 - \rho_k)^{i-1} / g_k^{'}(1)$. Therefore, conditionally on $\geom_k$, with probability $1 - \mathcal{O}(N^{-\lambda})$, the size of these components can be coupled with the size of $\geom_k$ i.i.d. Galton-Watson trees independent of $\geom_k$ and whose reproduction laws have generating series given by $\tilde g_{k}(s) := \hat{g}_k((1-\rho_k)s) / (1 - \rho_k)$. Therefore, the expected size of a component is given by:
\begin{equation*}  
\E \left[  \left| W^{(1)} \right| \, \Big| \, \mathcal{F}_k  \right] = \frac{1}{1 - \tilde g_{k}'(1) } + \mathcal{O}( N^{-\lambda})  = \frac{1}{1 - \hat{g}_k'(1-\rho_k)  } + \mathcal{O}( N^{-\lambda}),
\end{equation*}
and we obtain, using Equation \eqref{eq: ExpecGeom}:
\begin{align}
\E \left[ T_{k+1} - T_k \, \Big| \, \mathcal{F}_k \right] &= 1 + 2 \times \E \left[ \sum\limits_{p=1}^{\mathfrak{e}_k}  \left| W^{(i)} \right| \, \Big| \, \mathcal{F}_k \right] \nonumber \\
                                                           &= 1 + 2 \left( \frac{1-\rho_k}{\rho_k} \left( 1- \hat{g}'_k(1-\rho_k)  \right) + \mathcal{O}(N^{-\lambda}) \right) \left( \frac{1}{1 - \hat{g}_k'(1-\rho_k)  } + \mathcal{O}( N^{-\lambda}) \right)  \nonumber \\
                                                           &= \frac{2-\rho_k}{\rho_k}  + \mathcal{O}\left( N^{-\lambda} \right) \nonumber \\
                                                           &= f \left( \frac{N_0(k)}{N} , \frac{N_1(k)}{N} , \ldots  \right) + \mathcal{O}( N^{-\lambda})
\label{eq: evolution of T_k}
\end{align}
which is the desired result for the evolution of $(T_k)$.

\bigskip

We now turn to the evolution of the $(N_i(k))$ which follows from the analysis of the expectation of the terms \eqref{eq2} and \eqref{eq3}. The term \eqref{eq2} accounts for the vertices of degree $i$ in the graph $\mathscr S_k \setminus \mathscr S_{k+1}$. Among these vertices, the vertex $v_k$ has a special role because it is conditioned to be matched to an element of $\mathbf{Surv}^k$. Therefore, we write
\[
V_i(\mathscr S_k \setminus \mathscr S_{k+1}) = \mathbf 1_{\{\deg_{\mathscr S_k} (v_k) =i\}} + \sum_{j=1}^{\mathfrak e_k} \sum_{v \in W^{(j)} } \mathbf 1_{\{\deg_{\mathscr S_k} (v) =i\}}.
\]
We first compute the expectation of the sum in the right hand side of the previous equation.
The connected components $W^{(1)}, \ldots, W^{(\geom_k)}$ are well approximated by independent Galton-Watson trees with offspring distribution given by $\hat{g}_n$, conditioned on extinction. Let $C_i$ be the number of individuals that have $i-1$ children in such a tree. These individuals all have degree $i$ in $\mathscr S_k$ and contribute to the sum. The quantity $C_i$ satisfies the following recursion established by summing over the possible number of children of the root:
\[ \E[C_i] = \E \left[ \sum\limits_{l \geq 0} \hat{p}_l (1-\rho_k)^{l}  \left( lC_i + \delta_{l=i-1}  \right) \right] = \E[C_i] \hat{g}_k'(1-\rho_k) + \hat{p}_{i-1}(1-\rho_k)^{i-1},  \]
which leads to
\begin{equation} \label{eq:Ci}
\E[C_i] = \frac{\hat{p}_{i-1} (1- \rho_k)^{i-1}}{1 - \hat{g}_k'(1- \rho_k)}.
\end{equation}
Therefore, multiplying \eqref{eq: ExpecGeom} and \eqref{eq:Ci}, we obtain
\begin{equation} \label{eq:Vi1}
\E \left[ V_i ({\mathscr{S}_k} \setminus \mathscr{S}_{k+1}) \, \Big| \, \mathcal{F}_k   \right] = \mathbb P \left( \deg_{\mathscr{S}_k} (v_k) = i \, \Big| \, \mathcal{F}_k  \right) + \frac{\hat{p}_{i-1}}{\rho_k}\left( 1 - \rho_k \right)^{i-1} + \mathcal{O}\left( N^{-\lambda}  \right).
\end{equation}
Note that the sum over $i$ of these terms gives the total number of vertices in the connected components associated to the first $\geom_k$ children of $v_k$: $(1-\rho_k)/ \rho_k + o(1)$. This is in agreement with Equation \eqref{eq: evolution of T_k}.

\bigskip

For the last term \eqref{eq3}, we use the fact that, with probability $1 - \mathcal{O}( N^{-\lambda})$, the elements of $\mathbf m_{v_k}$ that belong to $\mathscr S_{k+1}$ are distinct.
One of these elements is $v_{k+1}$ and has a special role, while all the others correspond to a uniform matching to a half edge of a vertex of $\mathscr S_{k+1} \setminus \{v_{k+1}\}$ and therefore have degree $i$ with probability $\hat p _{i-1}$. Note that there are $\deg_{\mathscr{S}_k}(v_k) - \mathfrak e_k - 1$ terms in the sum \eqref{eq3} when excluding $v_{{k+1}}$. We have, taking into account that $o_v =1$ up to a negligible term:
\begin{align} \label{eq:Vi2}
\E & \left[
\sum_{v \,\in  \, \mathbf m_{v_k} \cap \mathscr S_{k+1}}
\left( - \mathbf{1}_{ \deg_{\mathscr{S}_k}(v) =i } + \mathbf{1}_{ \deg_{\mathscr{S}_k}(v)= i+ o_v } \right)
\right] \nonumber \\
             & \quad = 
             - \mathbb P \left( \deg_{\mathscr{S}_k} (v_{k+1}) = i \, \Big| \, \mathcal{F}_k  \right)
             + \mathbb P \left( \deg_{\mathscr{S}_k} (v_{k+1}) = i + 1 \, \Big| \, \mathcal{F}_k  \right) \nonumber\\
             & \quad \quad \quad +
              \E\left[ \deg(v_k) - \geom_k -1 \, \Big| \, \mathcal{F}_k \right]
              \left( -\hat p _{i-1} + \hat p _{i}\right)
              + \mathcal{O}( N^{-\lambda} ) \nonumber \\
             & \quad = 
             - \mathbb P \left( \deg_{\mathscr{S}_k} (v_{k+1}) = i \, \Big| \, \mathcal{F}_k  \right)
             + \mathbb P \left( \deg_{\mathscr{S}_k} (v_{k+1}) = i + 1 \, \Big| \, \mathcal{F}_k  \right) \nonumber\\
             & \quad \quad \quad +
             \frac{1}{\rho_k} \Bigg[ \Big( \hat{g}_k'(1) - (1- \rho_k)\hat{g}_k'(1-\rho_k) - (1-\rho_k)(1-\hat{g}_k'(1-\rho_k) )- \rho_k \Big) \nonumber  \\ 
             & \hspace{7.5cm} \times \left( - \hat{p}_{i-1} + \hat{p}_i \right) \Bigg] + \mathcal{O}( N^{-\lambda} ) \nonumber \\
             & \quad = 
             - \mathbb P \left( \deg_{\mathscr{S}_k} (v_{k+1}) = i \, \Big| \, \mathcal{F}_k  \right)
             + \mathbb P \left( \deg_{\mathscr{S}_k} (v_{k+1}) = i + 1 \, \Big| \, \mathcal{F}_k  \right) \nonumber\\
             & \quad \quad \quad +
             \frac{1}{\rho_k} \left(1- \hat{g}_k'(1) \right)\left( \hat{p}_{i-1} - \hat{p}_i \right) + \mathcal{O}( N^{-\lambda} ).
\end{align}

Hence, summing \eqref{eq:Vi1} and \eqref{eq:Vi2}, we obtain the total contribution of \eqref{eq2} and  \eqref{eq3}:
\begin{align*} 
\E \left[ N_i(k+1) - N_i(k) \, \Big| \, \mathcal{F}_k   \right] &= 
- \mathbb P \left( \deg_{\mathscr{S}_k} (v_k) = i \, \Big| \, \mathcal{F}_k  \right) 
- \mathbb P \left( \deg_{\mathscr{S}_k} (v_{k+1}) = i \, \Big| \, \mathcal{F}_k  \right) \\
& \qquad + \mathbb P \left( \deg_{\mathscr{S}_k} (v_{k+1}) = i + 1 \, \Big| \, \mathcal{F}_k  \right) \nonumber\\
& \qquad - \frac{\hat{p}_{i-1}}{\rho_k}\left( 1 - \rho_k \right)^{i-1} + \frac{1}{\rho_k}\left(1- \hat{g}_k'(1) \right)\left( \hat{p}_{i-1} - \hat{p}_i  \right) + \mathcal{O}( N^{-\lambda} ).
\end{align*}
Recall that the conditional law of $\deg_{\mathscr{S_k}}(v_k)$ is given by equation \eqref{eq: DegSurv}. Similar arguments to those used to compute it lead to
\[
\mathbb P \left( \deg_{\mathscr{S}_k} (v_{k+1}) = i \, \Big| \, \mathcal{F}_k  \right) = \mathbb P \left( \deg_{\mathscr{S}_k} (v_{k}) = i - 1 \, \Big| \, \mathcal{F}_k  \right) + \mathcal O (N^{-\lambda}).
\]
Therefore, we have
\begin{align*}
\E \left[ N_i(k+1) - N_i(k) \, \Big| \, \mathcal{F}_k   \right] &= - \frac{\hat{p}_{i-1}}{\rho_k} \left( 1 - \left( 1 - \rho_k \right)^{i-1} \right)
- \frac{\hat{p}_{i-1}}{\rho_k}\left( 1 - \rho_k \right)^{i-1} \\
& \qquad + \frac{1}{\rho_k}\left(1- \hat{g}_k'(1) \right)\left( \hat{p}_{i-1} - \hat{p}_i  \right) + \mathcal{O}( N^{-\lambda} ) \\
&= - \frac{\hat{p}_{i-1}}{\rho_k} + \frac{1}{\rho_k}\left(1- \hat{g}_k'(1) \right)\left( \hat{p}_{i-1} - \hat{p}_i  \right) + \mathcal{O}( N^{-\lambda} ) \\
&= f_i \left( \frac{N_0(k)}{N} , \frac{N_1(k)}{N} , \ldots  \right) + O \left( N^{-\lambda} \right). 
\end{align*}
This ends the proof of lemma \ref{lemma:trends}.
\end{proof}

We no turn to the proof of our main Theorem.

\begin{proof}[Proof of Theorem \ref{theo: fluidlimits}.] 
Let $t_\varepsilon = \sup \{ t \in [0,1], \, \lfloor tN \rfloor \leq K_\varepsilon \}$.
Let $\beta \in (0, 1/2)$ be as in Lemma \ref{lemma:trends}. Let $\nu \in \left(1, \frac{1-\beta}{\beta}  \right)$ and $w = N^{(1+\varepsilon)\beta}$. Fix $\eta>0$ and $I_\eta \geq 1$ such that for all $t\in[0,1)$:
\[  \max \left\{ \sum\limits_{i \geq I_\eta} z_i^*(t), \sum\limits_{i \geq I_\eta} iz_i^*(t) , \sum\limits_{i \geq I_\eta} i(i-1) z_i^*(t)   \right\} \leq \eta  \]
and such that w.h.p. 
\[  \max \left\{ \sum\limits_{i \geq I_\eta} N_i(\lfloor tN \rfloor), \sum\limits_{i \geq I_\eta} iN_i(\lfloor tN \rfloor) , \sum\limits_{i \geq I_\eta} i(i-1) N_i(\lfloor tN \rfloor)   \right\} \leq \eta N.  \] 
Observe that such a $I_\eta$ exists since, by monotonicity, it is enough to check the inequalities at $t=0$.
We prove by induction on $k$ that there exists a nondecreasing sequence $(B_k)_{0 \leq k \leq N/w}$ and $\delta>0$ such that $\limsup_N B_{N/w} / N \leq C \eta$ for some constant $C>0$ and such that
\begin{equation}\label{eq:recurrence_ass}
\forall 0 \leq i \leq I_\eta, \, \forall 0 \leq k \leq N/w, \quad \mathbb{P}\left( \left| N_i(kw) - z_i^*(kw/N)N \right| \geq B_k \right) \leq \exp( - N^\delta).
\end{equation}
%Notice that by Assumption \eqref{assump: degree} $I_\eta^3 \eta \rightarrow 0$ as $\eta \rightarrow 0$, which ensures that \eqref{eq:recurrence_ass} implies Theorem \ref{theo: fluidlimits}.

Since the sequence of configuration models has asymptotic degree distribution $\bm \pi$ and since $z_i^*(0) = {\bm \pi}_i$, there exists $N_\eta \geq 1$ such that $| N_i(0)/N - z_i^*(0)| \leq \eta$ for all $N \geq N_\eta$. This proves the initialization step.\\
\noindent Suppose that the property is verified for $0 \leq k \leq N/w -1$. Rewrite
\begin{align}
N_i\left((k+1)w\right) - z_i^*(&(k+1) w  / N) N \nonumber  \\
								&=  N_i\left((k+1)w\right) - N_i(kw) - w f_i\left((N_j(kw)/ N)_j\right) \label{eq: troncon} \\
                                &+  N_i(kw) - z_i^*(kw/N) N  \label{eq: HR} \\
                                &+ z_i^*(kw / N)N - z_i^*((k+1)w / N )N + wf_i\left((z_j^*(kw/N))_j \right) \label{eq: Taylor eq diff} \\
                                &+ wf_i\left((N_j(kw)/N)_j \right) -  wf_i\left((z_j^*(kw/N))_j \right)   \label{eq: Lipsch}
\end{align}
We analyse each term separately. \\
{\bf The term \eqref{eq: troncon}.} Let $\alpha \in \left( \frac{1+\nu}{2}\beta, \nu \beta  \right)$. Then, there exists $\lambda'>0$ such that with high probability,
\begin{equation} \label{eq:ProofWormald}
| \eqref{eq: troncon} | \leq N^{\alpha + \beta} + N^{(1+\nu)\beta - \lambda'}.
\end{equation}
Indeed, by the trend assumption, there exists a function $g(N)$ such that $g(N) = O( N^{ - \lambda })$ and such that the process
\[  \{ N_i\left(k w + l \right) - Y(kw) - l f_i\left(N_j(kw)/ N)_j\right) - lg(N) \}_{1 \leq l \leq w}  \]
is a supermartingale with increments bounded by $N^\beta$. Using Azuma-Hoeffding inequality with $l=w$, this implies that:
\begin{multline*}
\mathbb{P} \Bigg(  \big(N_i\left((k+1)w\right) - N_i(kw)  -  w f_i\left(N_j(kw)/ N)_j\right)  \big)  >  \left( N^{\alpha + \beta} + w  g(N) \right) \Bigg)
\\ \leq \exp \left( - \frac{1}{2w} \frac{ N^{2\alpha + 2\beta} }{N^{2 \beta}}  \right).
\end{multline*}
Since $\lambda'<\lambda$ and since $wg(N) = O ( N^{(1+\nu)\beta) - \lambda})$, we have proved that:
\begin{align*}
\mathbb{P} \Bigg(  \big(N_i\left((k+1)w\right) - N_i(kw)  -   w f_i\left(N_j(kw)/ N)_j\right) \big)  >  \big( N^{\alpha + \beta} + & N^{(1+\nu) \beta - \lambda'}  \big) \Bigg)  \\
&\leq \exp \left( - \frac{ N^{2 \alpha - (1 + \nu)\beta} }{2}   \right).
\end{align*} 
Using a similar argument, one can obtain the same bound on the probability that $\big( N_i\left((k+1)w\right) - N_i(kw)  -   w f_i\left(N_j(kw)/ N)_j\right)   \big)  < \left( N^{\alpha + \beta} + N^{(1+\nu)\beta - \lambda'}\right)$ and thus obtain inequality \eqref{eq:ProofWormald}.

\noindent {\bf The term \eqref{eq: HR}.} By our induction hypothesis, it can be bounded by $B_k$ with high probability.

\noindent {\bf The term \eqref{eq: Taylor eq diff}.} Using that $z_i^*$ is a solution of $z_i'(t) = f_i\left((z_j(t))_j\right)$ and the mean value Theorem, there exists $\theta \in [kw/N, (k+1)w/N]$ such that
\begin{equation}
|\eqref{eq: Taylor eq diff}| = w |f_i\left((z_j^*(\theta))_j\right) - f_i\left((z_j^*(kw/N))_j\right) |.
\end{equation}
Since $t \mapsto f_i( (z_j^*(t))_j)$ is smooth, we get that there exists a constant $C>0$ such that for every $i \leq I_\eta$:
\begin{equation}
| \eqref{eq: HR} | \leq C \frac{w^2}{N}.
\end{equation}

\noindent {\bf The term \eqref{eq: Lipsch}.} By our induction hypothesis and by our choice of $I_\eta$, there exists a constant $C>0$ such that w.h.p.
\begin{align*}
\left| \sum\limits_{j\geq 0} j N_j(kw)/N  - \sum_{j \geq 0} j z_j^*(kw/N) \right| &\leq C I_\eta^2 B_k/N + \sum\limits_{j>I_\eta} j z_j^*(kw/N) + \sum\limits_{j>I_\eta} j N_j(kw)/N \\
&\leq  C I_\eta^2 B_k/N  + 2 \eta.
\end{align*}
Similarly:
\[   \left| \sum\limits_{j\geq 1} j(j-1) N_j(kw)/N  - \sum\limits_{j\geq 1} j (j-1) z_j^*(kw/N) \right| \leq   C I_\eta^3 B_k/N + 2 \eta. \]
and
\[ \left| \rho_{ (N_j(kw)/N)_j } - \rho_{ (z_j^*(kw/N))_j } \right| \leq C I_\eta^3 B_k/N + 2 \eta. \]
Moreover, the quantities $\rho_{(z_j^*(t))_j}$ and $\sum_{j \geq 0} j z_j^*(t)$, which appear in the denominator of $f_i$, are bounded away from $0$ for all $t \in (0, t_\varepsilon)$ by our choice of $K_\varepsilon$.
Therefore, it can be checked that with high probability, there exists a constant $C$ such that
\begin{align*}
\left|  \eqref{eq: Lipsch}  \right| &= \left| wf_i\left((N_j(kw)/N)_j \right) -  wf_i\left((z_j^*(kw/N))_j \right)   \right| \\
&\leq  C w \left[  \left( \left| \frac{N_i(kw)}{N} - z_i^*(kw/N) \right| + \left| \frac{N_{i+1}(kw)}{N}- z_{i+1}^*(kw/N) \right| \right) + I_\eta^3 B_k/N + \eta \right] \\
                                    &\leq 
                                    \left\{ 
                                    \begin{array}{lr}
                                     C w ( 2   B_k/N + I_\eta^3 B_k/N + \eta) & \text{if $i \leq I_\eta-1$}, \\
                                     C w (  B_k/N + I_\eta^3 B_k/N + 3 \eta )  & \text{otherwise},
									\end{array}                                      
									\right.
\end{align*}
the case $i \leq I_\eta -1$ resulting from our induction hypothesis, and the case $i=I_\eta$ from the definition of the truncation index $I_\eta$.

\noindent \textbf{Conclusion.} Putting all previous arguments together, we deduce that there exists a constant $C>0$ such that, by taking
\begin{equation}\label{eq:rec_on_Bk}
B_{k+1} = B_k \left( 1 + C I_\eta^3 \frac{w}{N}  \right) + N^{\alpha + \beta} + N^{(1+\nu)\beta - \lambda'} + C N^{2(1+\nu)\beta - 1} + C w \eta, 
\end{equation} 
the following inequality holds:
\begin{equation*}
\forall 0 \leq i \leq I, \quad \mathbb{P}\left( \left| N_i((k+1)w) - z_i^*((k+1)w/N)N \right| \geq B_{k+1} \right) \leq \exp( - N^\delta).,
\end{equation*}
which concludes the heredity argument of the induction. \\
Finally, notice that from our calibration of the constants $\alpha, \beta$, the main additive term in \eqref{eq:rec_on_Bk} is the last one of order $w$. On the other hand, the multiplicative term gives a contribution of order $1$ after $N/w$ steps. Therefore, $\limsup_N B_{N/w} / N \leq C \eta$, ending the proof of Theorem \ref{theo: fluidlimits}.
\end{proof}

%IL FAUT UTILISER :
%\[ C = I_\eta^3 \]
%\[ \limsup B_{N/w}/N \leq C I_\eta^3 \eta  \]
%Donc il faut ajouter l'hyppothese: ${\bm \pi_i} \sim i^{-4}$.

\section{Proofs of the main results}

We now turn to the proofs of Proposition \ref{prop:law} and Theorem \ref{th:profile}. We will use the following general fact about contour processes of trees, which can be easily proved by induction on $n$.
\begin{equation}\label{eq:GeneralFactContourProcess}
\forall n \geq 0, \quad  \text{number of vertices explored by the DFS by time n} = \frac{n + X_n}{2}.
\end{equation}

\subsection{Proof of Proposition \ref{prop:law}}
The time variable in Proposition \ref{prop:law} is the proportion of vertices explored by the DFS whereas in Theorem \ref{theo: fluidlimits} it is the index of the ladder times $T_k$. Therefore, to prove Proposition \ref{prop:law}, a first step is to study the asymptotic proportion of vertices explored by time $T_k$. By Equation \eqref{eq:GeneralFactContourProcess}, for all $N \geq 1$ and all $1 \leq k \leq K_\varepsilon$, this proportion is given by $\omega(T_k) := \frac{k+T_k}{2 N}$. Therefore, by Theorem \ref{theo: fluidlimits}, this proportion satisfies
\begin{equation} \label{eq:ztilde}
\omega(T_k) = \widetilde{z} \left( \frac{k}{N} \right) + o(1), \quad \quad \text{with} \quad \widetilde{z} (t) = \frac{1}{2}\left( t + z\left(t \right) \right). 
\end{equation}

Fix $0 \leq \alpha < \alpha_c$ and recall the definition of $\tau^{(N)}(\alpha)$ given in Proposition \ref{prop:law}. At time $T_{N\widetilde{z}^{-1}(\alpha)}$, by Equation \eqref{eq:ztilde}, the number of explored vertices is $\alpha N + o(N)$. Therefore $\tau^{(N)}(\alpha) = T_{N\widetilde{z}^{-1}(\alpha)} + o(1)$. Hence, for all $i\geq 0$,
\begin{align*}
N_i(\tau^{(N)}(\alpha)) 
&= N_i\left( T_{N\widetilde{z}^{-1}(\alpha)} + o(1)  \right) \\
&= N z_i \left( \widetilde{z}^{-1}(\alpha)  \right) + o(N).
\end{align*}
It is easy to check that the sequence of functions $(z_i \circ \widetilde{z}^{-1})_{i \geq 0}$ is solution of the system (S') of Lemma \ref{lemme:EqDiff2} below. The generating function $g(\alpha,s)$ of Proposition \ref{prop:law} is given by 
\[  g(\alpha,s) = \frac{1}{1-\alpha} \sum\limits_{i \geq 0} z_i \circ \widetilde{z}^{-1}(\alpha) s^i,  \]
which is the desired result by Equation \eqref{eq: SolSerGen} and Proposition \ref{prop:eqdif}.

\subsection{Proof of Theorem \ref{th:profile}}
Let $N\geq 1$. By definition, for all $1 \leq k \leq K_\varepsilon$, the contour process of the tree constructed by the DFS algorithm at time $T_k$ is located at point $(T_k,k)$. Furthermore, by Theorem \ref{theo: fluidlimits}, 
\[  (T_k,k) = N \left( z\left( \frac{k}{N}\right) + o(1), \frac{k}{N} \right) .   \]
Note that $|T_{k+1} - T_k| = o(N)$ and that, between two consecutive $T_k$'s, the contour process cannot fluctuate by more than $o(N)$. Hence, after normalization by $N$, the limiting contour process converges to the curve $(z(t),t)$ where $t$ ranges from $0$ to $t_{\max} = \sup \{ t > 0, \, z'(t) < +\infty   \}$. Recall that by the definition of $z$ in Theorem \ref{theo: fluidlimits} and Equation \eqref{eq:deff}, $z'(t)=(2-\rho_{(z_i(t))_{i\geq 0}})/\rho_{(z_i(t))_{i\geq 0}}$. Hence, if we parametrize $(z(t),t)$ in terms of $\rho = \rho_{(z_i(t))_{i\geq 0}}$, the curve can be written $(x(\rho),y(\rho))$ where the functions $x$ and $y$ satisfy
\[  \frac{x'(\rho)}{y'(\rho)}  = \frac{2 - \rho}{\rho} . \]
Note that when $t$ ranges from $0$ to $t_{\max}$, the parameter $\rho$ decreases from $\rho_{\bm \pi}$ to $0$. In order to get a second equation connecting $x'$ and $y'$, we go back to the discrete process and observe that, by Equation \eqref{eq:GeneralFactContourProcess}, the number of explored vertices at time $T_k$ is equal to $(k+T_k)/2$. Using the notation of Proposition \ref{prop:law}, let $\hat{g}(\alpha,\cdot)$ be the size-biased version of $g(\alpha,\cdot)$. For all $\rho \in (0, \rho_{\bm \pi}]$, let $\alpha(\rho)$ be the unique solution of $1-\rho=\hat{g}(\alpha(\rho),1-\rho)$. After renormalizing by $N$, we get that:
\[  \frac{x(\rho) + y(\rho)}{2} = \alpha(\rho). \]
This yields the following system of equations:
\[
\begin{cases}
\frac{x'(\rho)}{y'(\rho)}  = \frac{2 - \rho}{\rho}  \\
\frac{x'(\rho) + y'(\rho)}{2} = \alpha'(\rho).
\end{cases}
\]
Therefore,
\[ 
\begin{cases}
x'(\rho) = (2-\rho) \alpha'(\rho) \\
y'(\rho) = \rho \alpha'(\rho)  .
\end{cases}
\]
Integrating by parts, this gives the formulas for $x^\uparrow$ and $y^\uparrow$ in Theorem \ref{th:profile}. Fix $\rho \in (0, \rho_{\bm \pi}]$. Then, the asymptotic profile of the decreasing phase of the DFS is obtained by translating horizontally each point $(x^\uparrow(\rho),y^\uparrow(\rho))$ of the ascending phase  to the right by  twice the asymptotic proportion of the giant component of the remaining graph of parameter $\rho$, which is $2(1-g(\alpha(\rho),1-\rho))$. Indeed, the time it takes to the DFS to return at a given height $k$ attained during the ascending phase corresponds to the time of exploration of the giant component of the unexplored graph at time $T_k$. The latter is given by twice the number of vertices of the giant component which is equal to $2 ( 1 - g_k(1 - \rho_k) )$.

\section{Technical lemmas} % (fold)
\subsection{Asymptotic densities in a configuration model}\label{sub:technical_lemmas}
In this section we establish Lemma \ref{lemme: Bollobas estimates}. The proofs of each of the four estimates follow the same scheme, therefore we only focus on the proof the last one, namely that there exists $\lambda >0$ such that:
\begin{equation*}
\PP\left( \left| \frac{|\mathbf{Surv}(\mathbf C_{{\mathbf{N}}(k)}^d)|}{2|E(\mathbf C_{{\mathbf{N}}(k)}^d)|} - \rho_k   \right| \geq N^{-\lambda}   \right)= \mathcal O \left( N ^{-1-\lambda}\right).
\end{equation*}
First, notice that for the values of $k$ that we consider and under our assumptions \eqref{assump: moment} and \eqref{assump: degree}, the number of edges and vertices of the graphs $\mathbf C_{{\mathbf{N}}(k)}^d$ are all of order $N$. Therefore, it is enough to prove the following bound:
\begin{equation*}
\PP\left( \left| \frac{|\mathbf{Surv}(\mathbf C_{{\mathbf{N}}(k)}^d)|}{2|E(\mathbf C_{{\mathbf{N}}(k)}^d)|} - \rho_k   \right| \geq |E(\mathbf C_{{\mathbf{N}}(k)}^d)|^{-\lambda}   \right)= \mathcal O \left( |E(\mathbf C_{{\mathbf{N}}(k)}^d)| ^{-1-\lambda}\right).
\end{equation*}
This is a direct consequence of the two following Lemmas. The first one is a general concentration result for the configuration model.
\begin{lemma}\label{lemme:concentration}
Fix $\gamma > 2$ and $n \geq 1$. Let $\mathbf d = (d_1, \ldots, d_n)$ be such that $\max \{d_1,\ldots, d_n \} \leq n^{1/\gamma}$.  Fix also $\delta \in (0,1/2)$ and recall that, for a graph $G$, $\mathbf{Surv}(G)$ denotes the set of half edges of $G$ attached to a vertex $v$ such that the connected component of $v$ after removing this half edge has at least $n^\delta$ vertices. Let $m = \sum_i d_i$ the number of half edges of a configuration graph $\mathcal C (\mathbf d)$, then, for any $\delta' \geq \delta$ one has
\[
\mathbb P \left( \left| \frac{|\mathbf{Surv}(\mathcal C (\mathbf d))|}{m} -  \frac{\mathbb E \left(|\mathbf{Surv}(\mathcal C (\mathbf d)) \right)|}{m} \right| \geq \frac{n^{\delta' + \frac{1}{\gamma}}}{2 \sqrt m }\right) \leq C \exp\left(- C n^{2(\delta' - \delta)}\right).
\]
\end{lemma}
The second Lemma consists in an estimation of the expectation of $|\mathbf{Surv}( \mathcal C (\mathbf d^{(n)}))|$ for a sequence of configuration models that satisfy the assumptions of Proposition \ref{prop:law}.
\begin{lemma}\label{lemme:EstimateExp}
Let $(\mathcal C ( \mathbf{d}^{(n)} ) )_{n \geq 1}$ be a sequence of configuration models with asymptotic degree distribution $\bm \pi$. We suppose that $\bm \pi$ is supercritical in the sense of Definition \ref{defn:supercritical} and that the sequence $\mathbf{d}^{(n)}$ satisfies assumption \eqref{assump: moment} and \eqref{assump: degree}.  

For all $n \geq 1$, let $g_n$ be the generating series associated to the empirical distribution of the degree sequence $\mathbf{d}^{(n)}$. Let $\rho_n$ be the smallest positive solution of the equation $\hat{g}_n(1-x)=1-x$. Then, for $n$ sufficiently large:
\[
\frac{\mathbb{E}\left[ |\mathbf{Surv}( \mathcal C (\mathbf d^{(n)}))| \right] }{2 g_n'(1)} = \rho_n + \mathcal{O}\left( n^{ 2 \delta + \frac{1}{\gamma} - 1} \right).
\]
\end{lemma}

\begin{proof}[Proof of Lemma \ref{lemme:concentration}.]
In order to prove Lemma \ref{lemme:concentration}, it is sufficient to check that the function $\mathbf{Surv}(\cdot)$ is Lipschitz in the following sense. We say that two configuration models are related by a switching if they differ by exactly two pairs of matched half-edges (see Figure \ref{fig:switch}). Then, we claim that $\mathbf{Surv}(\cdot)$ is such that, for any two graphs $\mathrm{G}_1$ and $\mathrm{G}_2$ differing by a switching:
\begin{equation} \left|  | \mathbf{Surv}(\mathrm{G}_1) | - | \mathbf{Surv}(\mathrm{G}_2) | \right| \leq 8  n^{ \delta + \frac{1}{\gamma}}.
\label{eq: Lipshitz switching}
\end{equation} 
Using a result of Bollob\'as and Riordan \cite[Lemma 8]{MR3343756}, this regularity implies the following concentration inequality:
\begin{equation} \PP\left( \left| | \mathbf{Surv}(\mathbf C_{{\mathbf{N}}(k)}^d) | - \E[ | \mathbf{Surv}(\mathbf C_{{\mathbf{N}}(k)}^d) | ]   \right| \geq t  \right) \leq 2 \exp \left( \frac{-t^2}{C n^{2\delta + \frac{2}{\gamma}} m}  \right),
\label{eq:Bollobas concentration}
\end{equation}

\begin{figure}[!h]
\centering
\includegraphics[height=2cm]{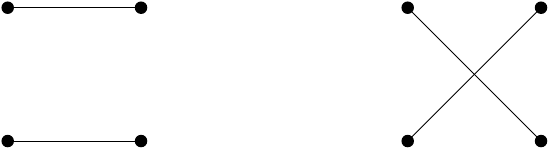}
\caption{\label{fig:switch}Switching two edges in a graph. }
\end{figure}

By taking $t=n^{\delta' + \frac{1}{\gamma}}  m^{\frac{1}{2}}$ in \eqref{eq:Bollobas concentration}, we obtain Lemma \ref{lemme:concentration}.

It remains to prove inequality \eqref{eq: Lipshitz switching}. To pass from $\mathrm{G}_1$ to $\mathrm{G}_2$, one has to delete two edges in $\mathrm{G}_1$ and then add two other edges. Therefore, it suffices to study the effect of adding an edge $e$ on a graph $\mathrm{G}$ having maximal degree $n^{1 / \gamma}$. Indeed, the effect of deleting an edge $f$ of a graph $\mathrm{H}$ is equal to the effect of adding the edge $f$ to the graph $\mathrm{H} \setminus \{f\}$.

Let $u$ and $v$ be the extremities of $e$. Let us define two partial orders associated respectively to $u$ and $v$ among the half-edges of $\mathbf{Ext}( \mathrm G) = \mathbf{Surv}(\mathrm G)^c$. We say that:
\begin{itemize}
\item $e_1 \preceq_u e_2$ if all the paths connecting $e_2$ to $u$ contain $e_1$,
\item $e_1 \preceq_v e_2$ if all the paths connecting $e_2$ to $v$ contain $e_1$.
\end{itemize}
Let $f_u$ (resp. $f_v$) be a maximal element for the partial order $\preceq_u$ (resp. $\preceq_v$), and denote by $\mathscr C_{f_u}$ (resp. $\mathscr C_{f_v}$) the connected component of the extremity of $f_u$ (resp. $f_v$) after the removal of $f_u$ (resp. $f_v$) in $\mathrm G$. Then, by maximality, the set of extremities of half-edges that change their status from $\mathbf{Ext}(\mathrm G)$ to $\mathbf{Surv}(\mathrm{G})$ after adding $e$ is included in $\mathscr C_{f_u} \cup \mathscr C_{f_v}$. See Figure \ref{fig:boll} for an illustration.
\begin{figure}[!h]
\centering
\includegraphics[scale=0.8]{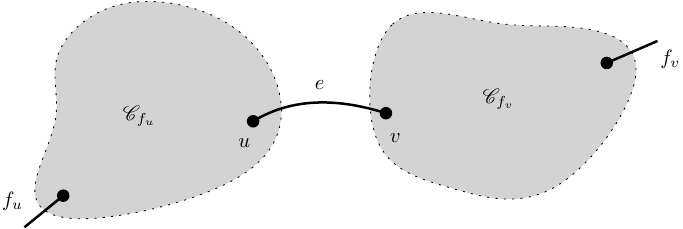}
\caption{\label{fig:boll} Effect of the edge $e$. }
\end{figure}
Since $f_u$ (resp. $f_v$) was in $\mathbf{Ext}(\mathrm{G})$, the number of vertices in $\mathscr C_{f_u}$ (resp. $\mathscr C_{f_v}$) is at most $n ^\delta$. Since the maximal degree of a vertex in $\mathrm{G}$ is $n^{1/\gamma}$, we deduce that:
\[  \left| | \mathbf{Surv}_n(\mathrm{G}) | - | \mathbf{Surv}_n(\mathrm{G} \cup e) | \right| \leq 2 n^{\delta + \frac{1}{\gamma}}. \]
This implies \eqref{eq: Lipshitz switching} and Lemma \ref{lemme:concentration}.
\end{proof}

\begin{proof}[Proof of Lemma \ref{lemme:EstimateExp}.]
Fix $n \geq 1$. Let $e$ be a uniformly chosen half-edge in $\mathcal C (\mathbf d^{(n)})$ and let $v$ be the extremity of $e$. We denote $\mathscr C_v$ the connected component of $v$ inside $\mathcal C (\mathbf d^{(n)}))$ after removing $e$. Then, since
$\mathbb{E}\left[ |\mathbf{Surv}( \mathcal C (\mathbf d^{(n)}))| \right] = 2g_n'(1) \mathbb{P}\left( e \in \mathbf{Surv}( \mathcal C (\mathbf d^{(n)})) \right)$, it is sufficient to prove that
\begin{equation}\label{eq:ProbSurv}
\mathbb{P}\left( | \mathscr C_v | \geq n^\delta \right) = \rho_n + \mathcal{O} \left( n^{2 \delta + \frac{2}{\gamma} - 1}  \right).
\end{equation}

Let $(d_i^{\uparrow})_{1 \leq i \leq n}$ and $(d_i^{\downarrow})_{1 \leq i \leq n}$ respectively denote the increasing and decreasing reordering of the degree sequence $(d_i)_{1 \leq i \leq n}$:
\[ d_1^{\uparrow} \leq \cdots \leq d_n^{\uparrow} \quad \quad \text{and} \quad \quad d_1^{\downarrow} \geq \cdots \geq d_n^{\downarrow}.  \] 

In order to prove \eqref{eq:ProbSurv}, we will use a coupling argument. More precisely, we first introduce two Galton-Watson trees: 
\begin{itemize}
\item $\mathscr T^-$ with reproduction law: $q_i^- := \frac{(i+1)| \{ j \geq \lceil n^\delta \rceil, \, d^{\downarrow}_j = i + 1  \} |}{\sum_{j \geq \lceil n^\delta \rceil} (j+ 1) d^{\downarrow}_j }$,
\item $\mathscr T^+$ with reproduction law: $q_i^+ := \frac{(i+1)| \{ j \geq \lceil n^\delta \rceil, \, d^{\uparrow}_j = i + 1  \} |}{\sum_{j \geq \lceil n^\delta \rceil} (j+ 1) d^{\uparrow}_j }$ .
\end{itemize}
We also let $E$ be the event where, in the $\lfloor n^\delta \rfloor$ first steps of the exploration of $\mathscr C_v$, a loop is discovered. Then, the following inequalities hold:
\begin{equation}\label{eq:couplinBounds}
\left(1- \PP(E)  \right)\PP( |\mathscr{T}^-| \geq n^\delta ) \leq \mathbb{P}\left( | \mathscr C_v | \geq n^\delta \right)  \leq \PP( |\mathscr{T}^+| \geq n^\delta ).
\end{equation}
Now, we prove that:
\begin{equation}\label{eq:T+size}
\begin{cases}
\PP( |\mathscr{T}^-| \geq n^\delta ) = \rho_n + \mathcal{O}(n^{\delta + \frac{1}{\gamma} -1}), \\
\PP( |\mathscr{T}^+| \geq n^\delta ) = \rho_n + \mathcal{O}(n^{\delta + \frac{1}{\gamma} -1}).
\end{cases}
\end{equation}
Since the proofs of these two bounds are similar, we only focus on the second one. Let $g_n^+(s)=\sum_{k \geq 0} q_k^+ s^k$ be the generating series of $(q_k^+)_{k\geq 0}$. Let $\rho_n^+$ be the smallest positive solution of $g_n^+(1-x)=1-x$. Then:
\begin{align}
\PP( |\mathscr{T}^+| \geq n^\delta ) 
&= \PP( |\mathscr{T}^+| = + \infty) + \PP( n^\delta \leq |\mathscr{T}^+| < + \infty) \nonumber \\
&= \rho_n^+ + o\left( \frac{1}{n} \right) \label{eq:T+Rho+}.
\end{align}

The difference between $\rho_n^+$ and $\rho_n$ can be written as follows:
\begin{align}
\rho_n^+ - \rho_n 
&= g_n^+(1-\rho_n^+) - g_n(1-\rho_n) \nonumber \\
&=  g_n(1-\rho_n^+) - g_n(1-\rho_n)  + g_n^+(1-\rho_n^+) - g_n(1-\rho_n^+) \nonumber \\
&= g_n'(1-\rho_n)(\rho_n - \rho_n^+) +  o\left( \rho_n^+ - \rho_n \right) + g_n^+(1-\rho_n^+) - g_n(1-\rho_n^+), \label{eq:Rho+MinusRho}
\end{align}
where in the last equality, we used a Taylor expansion. From the definition of $(q_k^+)_{k \geq 0}$, for all $k \geq 0$, it holds that:
\[  q_k^+ = p_k + \mathcal{O}\left(  \frac{n^{ \delta + \frac{1}{\gamma} }}{n} \right),  \]
where the error term is uniform in $k$. In particular, this implies that $g_n^+(1-\rho_n^+) - g_n(1-\rho_n^+)$ is of order $n^{\delta + \frac{1}{\gamma} -1}$. Inserting this into \eqref{eq:Rho+MinusRho}, we get
\[ \left( 1 - g_n'(1-\rho_n) + o(1)  \right)\left( \rho_n^+ - \rho_n \right) = \mathcal{O} \left( n^{\delta + \frac{1}{\gamma} -1} \right).   \]
By the assumptions of Lemma \ref{lemme:EstimateExp}, $\rho_n$ converges to the fixed point of $g_{\bm \pi}$, which is bounded away from $0$. Therefore, for large enough $n$, $g_n'(1-\rho_n)$ is bounded away from $1$. Hence
\begin{equation*}
|\rho_n^+ - \rho_n| = \mathcal{O} \left( n^{\delta + \frac{1}{\gamma} -1} \right).
\end{equation*}
Together with \eqref{eq:T+Rho+}, this implies \eqref{eq:T+size}.
\bigskip 

It remains to estimate the probability of the event $E$. During the first $\lfloor n^\delta \rfloor$ steps of the exploration of $\mathscr C_v$, the number of half-edges of the explored cluster is at most $n^\delta \times n^{1 / \gamma}$. Hence, the probability of creating a loop at each of these steps is of order $n^{\delta + \frac{1}{\gamma} - 1}$. Therefore, by the union bound:
\begin{equation}  \label{eq:ProbE}
\PP(E) = \mathcal{O} \left( n^{2 \delta + \frac{1}{\gamma} - 1}  \right) .
\end{equation}

Gathering \eqref{eq:couplinBounds}, \eqref{eq:T+size} and \eqref{eq:ProbE}, we get \eqref{eq:ProbSurv} and therefore Lemma \ref{lemme:EstimateExp}.
\end{proof}

\subsection{An infinite system of differential equations}\label{sec: EqDiff}
The aim of this section is to prove Lemma \ref{lemma: EqDiff}. In the following, we fix a probability distribution $\bm \pi = (\bm \pi_i)_{i \geq 0}$ which is supercritical in the sense of Definition \ref{defn:supercritical}.

First, we prove that the problem can be reduced to the study of another system of differential equations. Recall that, given a sequence $(\zeta_i)_{i \geq 0} \in \mathbb{R}^{\mathbb{Z}_+}$ such that $\sum_{i \geq 0} \zeta_i \leq 1$, the implicit quantity $\rho_{(\zeta_i)_{i \geq 0}}$ is defined through Equations \eqref{eq: DefnGenSer} and \eqref{eq:rho}.

\begin{lemma} \label{lemme:EqDiff2}
If the following system has a unique solution well defined on some maximal interval $[0,t_{\max}')$ for some $t_{\max}' > 0$:
\begin{equation}\tag{S'}\label{eq: EqDiff2}
\left\{ 
\begin{array}{lcl}
\frac{\mathrm{d}\zeta_i}{\mathrm{d}t} & = &  -  \frac{i \zeta_{i}}{\sum_{j \geq 0} j \zeta_j} + \frac{1}{\sum_{j \geq 0} j \zeta_j} \left( 1 - \frac{\sum_{j \geq 0} (j-1)j \zeta_j}{ \sum_{n \geq 0} j \zeta_j }  \right) \left( i \zeta_{i} - (i+1)\zeta_{i+1} \right) \\
\zeta_i(0) & = & \bm \pi_i,
\end{array}
\right.
\end{equation}
then the system \eqref{eq: EqDiff} has a unique solution well defined on a maximal interval $[0,t_{\max})$ for some $t_{\max}>0$.
\end{lemma}
\begin{proof}
Suppose that \eqref{eq: EqDiff2} has a unique solution $(\zeta_i)_{i \geq 0}$. Let $\phi$ be the unique function defined by 
\[ 
\begin{cases} 
 \phi'(t) \rho_{(\zeta_i(t))_{i\geq 0} } = 1 , \\
 \phi(0) = 0.
\end{cases}
\]
Then, for all $i \geq 0$, $(\zeta_i \circ \phi)'(t) = \frac{1}{\rho_{(\zeta_i(t))_{i\geq 0}} }  \times \rho_{(\zeta_i(t))_{i \geq 0}} f_i(\zeta_0(t), \zeta_1(t), \ldots ) =  f_i(\zeta_0(t), \zeta_1(t), \ldots )$ which proves that $(\zeta_i \circ \phi)_{i \geq 0}$ is a solution of the system \eqref{eq: EqDiff}.

Let $(z_i)_{ i \geq 0}$ be a solution of \eqref{eq: EqDiff}. Then, for all $t \geq 0$ where it is well defined,
\[  \sum\limits_{i \geq 0} z_i(t) = 1 - \int_0^{t}  \frac{1}{\rho_{ (z_i(t))_{i \geq 0} }} \mathrm{d}u =: 1 - \psi(t).   \]
Then, $(z_i \circ \psi^{-1} )_{i \geq 0}$ is a solution of \eqref{eq: EqDiff2}. Therefore, since $(\zeta_i)_{i \geq 0}$ is unique, $(z_i \circ \psi^{-1} \circ \phi)_{i \geq 0} = (\zeta_i \circ \phi)_{i \geq 0}$ is also solution of \eqref{eq: EqDiff}. In particular, this implies that
\[ \frac{-1}{ \rho_{  (z_i \circ \psi^{-1} \circ \phi (t))_{i \geq 0} }  }   = \frac{ \mathrm{d} }{ \mathrm{d} t} \left( \sum\limits_{i \geq 0}  z_i \circ \psi^{-1} \circ \phi  \right) (t) = \left( \psi^{-1} \circ \phi  \right)^{'} (t) \times  \frac{-1}{ \rho_{ (z_i \circ \psi^{-1} \circ \phi (t))_{i \geq 0} }  } .  \]
Therefore, $\psi = \phi$ only depends on $(\zeta_i)_{i \geq 0}$, yielding the uniqueness of the solution.
\end{proof}

We now exhibit a solution of \eqref{eq: EqDiff2}. Let $f_{ \bm \pi}(s) = \sum_{i \geq 0} \bm \pi_i s^i$ be the generating series associated to $\bm \pi$. Define $t_{\max}'$ to be the unique root  between $0$ and $1- {\bm \pi}_0$ of the equation
\[  \frac{f_{\bm \pi}'' \left(   f_{\bm \pi}^{-1} (1-t) \right) }{ f_{\bm \pi}'(1)} = 1. \]
For all $0 \leq t \leq t_{\max}'$ and $0\leq s \leq 1$, let
\begin{equation} \label{eq: SolSerGen} 
f(t,s) :=  f_{\bm \pi} \left(  f_{\bm \pi}^{-1}(1-t) - (1-s) \frac{f_{\bm \pi}'(f_{\bm \pi}^{-1}(1-t))}{f_{\bm \pi}'(1)} \right). 
\end{equation}
Note that this restriction to the interval $[0,t_{\max}')$ will play a crucial role in the analytic proof of the uniqueness of the solution. Moreover, from a probabilistic point of view, it corresponds to the range of times where $\frac{1}{1-t}f_{\bm \pi}(t,s)$ is the generating series of a supercritical probability law.

\begin{proposition} \label{prop:eqdif}
For all $0 \leq t \leq t_{\max}'$ and $i \geq 0$, let $\zeta_i(t) := [s^i]f(t,s)$ be the coefficient of $s^i$ in $f(t,s)$. Then $(\zeta_i)_{i \geq 0}$ is a solution of \eqref{eq: EqDiff2}.
\end{proposition}
\begin{proof}
It can be easily verified that $f(t,s)$ satisfies the following equation:
\[ \frac{\partial f}{\partial t} (t,s) = \frac{ \frac{\partial f}{\partial s} (t,s) }{\frac{\partial f}{\partial s} (t,1)} \left( (1-s) \frac{\frac{\partial^2 f}{\partial s^2} (t,1)}{\frac{\partial f}{\partial s} (t,1)} -1 \right). \]
By extracting the coefficient of $s^i$ we get that
\[ \frac{\mathrm{d}\zeta_i}{\mathrm{d}t} =  -  \frac{i \zeta_{i}}{\sum_{j \geq 0} j \zeta_j} + \frac{1}{\sum_{j \geq 0} j \zeta_j} \left( 1 - \frac{\sum_{j \geq 0} (j-1)j \zeta_j}{ \sum_{n \geq 0} j \zeta_j }  \right) \left( i \zeta_{i} - (i+1)\zeta_{i+1} \right),   \] 
which ends the proof the proposition.
\end{proof}

\begin{acks}
The authors are pleased to thank warmly an anonymous referee for its careful reading, suggestions, and pointing out a mistake in the original manuscript.
The first author would like to thank the ANR grants MALIN (Projet- ANR-16-CE93-0003) and PPPP (Projet-ANR-16-CE40-0016) for their financial support. The other three authors would like to thank the ANR grant ProGraM (Projet-ANR-19-CE40-0025) for its financial support. G.F. and L.M. also ackowledge the support of the Labex MME-DII (ANR11-LBX-0023-01).
\end{acks}

%%%%%%%%%%%%%%%%%%%%%%%%%%%%%%%%%%%%%%%%%%%%%%%%%%%%%%%%%%%%%%%%%%%
%%                                                               %%
%% You have reached the end of your document.                    %%
%%                                                               %%
%%%%%%%%%%%%%%%%%%%%%%%%%%%%%%%%%%%%%%%%%%%%%%%%%%%%%%%%%%%%%%%%%%%

\end{document}